\documentclass[11pt]{article}
\usepackage[centertags]{amsmath}
\usepackage{amsfonts}
\usepackage{amssymb}
\usepackage{amsthm,color}
\usepackage{newlfont}
\usepackage{bbm}
\usepackage{graphicx}

\newtheorem{Theorem}{Theorem}[section]
\newtheorem{Definition}[Theorem]{Definition}

\newtheorem{Lemma}[Theorem]{Lemma}
\newtheorem{Corollary}[Theorem]{Corollary}
\newtheorem{Remark}[Theorem]{Remark}

\numberwithin{equation}{section}

\begin{document}

\def\le{\left}
\def\r{\right}
\def\cost{\mbox{const}}
\def\a{\alpha}
\def\d{\delta}
\def\ph{\varphi}
\def\e{\epsilon}
\def\la{\lambda}
\def\si{\sigma}
\def\La{\Lambda}
\def\B{{\cal B}}
\def\A{{\mathcal A}}
\def\L{{\mathcal L}}
\def\O{{\mathcal O}}
\def\bO{\overline{{\mathcal O}}}
\def\F{{\mathcal F}}
\def\K{{\mathcal K}}
\def\H{{\mathcal H}}
\def\D{{\mathcal D}}
\def\C{{\mathcal C}}
\def\M{{\mathcal M}}
\def\N{{\mathcal N}}
\def\G{{\mathcal G}}
\def\T{{\mathcal T}}
\def\R{{\mathbb R}}
\def\I{{\mathcal I}}

\def\bw{\overline{W}}
\def\phin{\|\varphi\|_{0}}
\def\s0t{\sup_{t \in [0,T]}}
\def\lt{\lim_{t\rightarrow 0}}
\def\iot{\int_{0}^{t}}
\def\ioi{\int_0^{+\infty}}
\def\ds{\displaystyle}
\def\pag{\vfill\eject}
\def\fine{\par\vfill\supereject\end}
\def\acapo{\hfill\break}

\def\beq{\begin{equation}}
\def\eeq{\end{equation}}
\def\barr{\begin{array}}
\def\earr{\end{array}}
\def\vs{\vspace{.1mm}   \\}
\def\rd{\reals\,^{d}}
\def\rn{\reals\,^{n}}
\def\rr{\reals\,^{r}}
\def\bD{\overline{{\mathcal D}}}
\newcommand{\dimo}{\hfill \break {\bf Proof - }}
\newcommand{\nat}{\mathbb N}
\newcommand{\E}{\mathbb E}
\newcommand{\Pro}{\mathbb P}
\newcommand{\com}{{\scriptstyle \circ}}
\newcommand{\reals}{\mathbb R}

\def\Amu{{A_\mu}}
\def\Qmu{{Q_\mu}}
\def\Smu{{S_\mu}}
\def\H{{\mathcal{H}}}
\def\Im{{\textnormal{Im }}}
\def\Tr{{\textnormal{Tr}}}
\def\E{{\mathbb{E}}}
\def\P{{\mathbb{P}}}
\def\span{{\textnormal{span}}}
\title{SPDEs on narrow domains and on graphs: an asymptotic approach}
\author{Sandra Cerrai, Mark Freidlin\\
\vspace{.1cm}\\
Department of Mathematics\\
 University of Maryland\\
College Park\\
 Maryland, USA
}

\date{}

\maketitle

\begin{abstract}
We introduce here a  class of stochastic partial differential equations defined on a graph and we show how they are obtained as the limit of suitable stochastic partial equations defined in a narrow channel, as the width of the channel goes to zero. To our knowledge, this is the first time an SPDE on a graph is studied.
\end{abstract}

\section{Introduction}

Let $G$ be a bounded domain in $\mathbb{R}^2$, having a smooth boundary $\partial G$. We consider here the following stochastic partial differential equation (SPDE) in $G$, with Neumann boundary conditions
\begin{equation}
\label{intro1}
\le\{
\begin{array}{l}
\ds{\frac{\partial u_\e}{\partial t}(t,x,y)=\frac 12\,\frac{\partial^2 u_\e}{\partial x^2}(t,x,y)+\frac 1{2\e^2}\frac{\partial^2 u_\e}{\partial y^2}(t,x,y)+b(u_\e(t,x,y))+\frac{\partial w^Q}{\partial t}(t,x,y),}\\
\vs
\ds{\frac{\partial u_\e}{\partial \nu_\e}(t,x,y)=0,\ \ \ (x,y) \in\,\partial G,\ \ \ \ u_\e(0,x,y)=u_0(x,y).}
\end{array}\r.
\end{equation}
Here $w^Q(t)$ is a cylindrical Wiener process in $L^2(G)$ and $\nu_\e=\nu_\e(x,y)$ is the unit interior conormal at $\partial G$, corresponding to the second order differential operator
\[{\cal L}_\e=\frac 12\,\frac{\partial^2}{\partial x^2}+\frac 1{2\e^2}\frac{\partial^2 }{\partial y^2}.\] The functions $b$ and $u_0$ and the noise $w^Q(t)$ are assumed to be regular enough  so that equation \eqref{intro1} admits a unique mild solution for every $\e>0$ (see below for all details).

After an appropriate change of variables, equation \eqref{intro1} can be obtained from the equation
\begin{equation}
\label{intro2}
\le\{\begin{array}{l}
\ds{\frac{\partial v_\e}{\partial t}(t,x,y)=\frac 12\,\Delta v_\e(t,x,y)+b(v_\e(t,x,y))+\sqrt{\e}\,\frac{\partial w^{Q_\e}}{\partial t}(t,x,y),}\\
\vs
\ds{\frac{\partial v_\e}{\partial \hat{\nu}_\e}(t,x,y)=0,\ \ \ (x,y) \in\,\partial G_\e,\ \ \ \ v_\e(0,x,y)=u_0(x,y/\e),}
\end{array}\r.
\end{equation}
where $G_\e$ is the narrow domain $\{(x,y) \in\,\reals^2\,:\,(x,y/\e) \in\,G\}$ and $\hat{\nu}_\e(x,y)$ is the inward unit normal vector at $\partial G_\e$. Reaction-diffusion equations of the same type as \eqref{intro2}, with or without additional noise, arise, for example, in models for the motion of molecular motors. Actually, one of the possible ways to model Brownian motors/ratchets is to describe them as
particles  traveling along a designated track, and the designated track along which the molecule/particle
is traveling can be viewed as a tubular domain with many  {\em wings} added to it. To this purpose, see e.g. \cite{fhb}.

In this paper, we are interested in the limiting behavior of the solution $u_\e$ of equation \eqref{intro1}, as $\e\downarrow 0$.  For this purpose, suppose for a moment that the noisy term $\partial w^Q/\partial t$ is replaced by a regular enough function $h(t,x,y)$ and, for the sake of brevity,  assume $b(u)=0$. If $(X^\e(t),Y^\e(t))$ is the diffusion process in $G\cup \partial G$ governed by the operator ${\cal L}_\e$ inside $G$ and undergoing instantaneous reflections at $\partial G$, with respect to the co-normal associated with ${\cal L}_\e$, then, as is well known (see, for example, \cite{fred}), the solution of \eqref{intro1} can be written in the form
\begin{equation}
\label{variation-intro}
u_\e(t,x,y)=S_\e(t) u_0(x,y)+\int_0^t S_\e(t-s)h(s,\cdot)(x,y)\,ds,\end{equation}
where,  for any $\varphi \in\,B_b(G)$,
\[S_\e(t)\varphi(x,y)=\mathbb{E}_{(x,y)}f(X^\e(t),Y^\e(t)).\]

The process $(X^\e(t),Y^\e(t))$ has a slow component $X^\e(t)$ and a fast component $Y^\e(t)$, if $0<\e<<1$. This means that, before $X^\e(t)$ changes a little, the $y$-component of the process hits the boundary many times. This leads to an additional drift in the limit of the $x$-component of the process, due to the changing width of the domain and to the averaging of the function $h(s,X^\e(t-s),Y^\e(t-s))$. Moreover, for a given $x$, the intersection of the domain $G$ with the vertical line containing $(x,0)$ can consist of several connected components (see e.g. the intervals $l_1(x)$ and $l_2(x)$ in Figure 1). This leads to the fact that the slow component of the process $(X^\e(t),Y^\e(t))$ lives on the graph $\Gamma$ (see again Figure 1).

This graph, actually, {\em counts} all normalized ergodic invariant measures of the two-dimensional process $(\hat{X}(t),\hat{Y}(t))$ in $G\cup \partial G$, where $\hat{X}(t)=\hat{X}(0)=x$ and $\hat{Y}(t)$ is the one-dimensional Wiener process with instantaneous reflection on $\partial G$. The process $(\hat{X}(t),\hat{Y}(t))$, up to a time change, is our non perturbed system. Thus, the slow component of the perturbed system (the process $(X^\e(t),Y^\e(t))$ on $G\cup \partial G$) is the projection $\Pi^\e(t)=\Pi(X^\e(t),Y^\e(t))$ of $(X^\e(t),Y^\e(t))$ on the simplex of normalized invariant measures of the original system.
Moreover, the graph $\Gamma$ parametrizes extreme points of the simplex and any point of the simplex is a linear convex combination of the extreme points. 

\begin{figure}{h}
\centering
\includegraphics[scale=1]{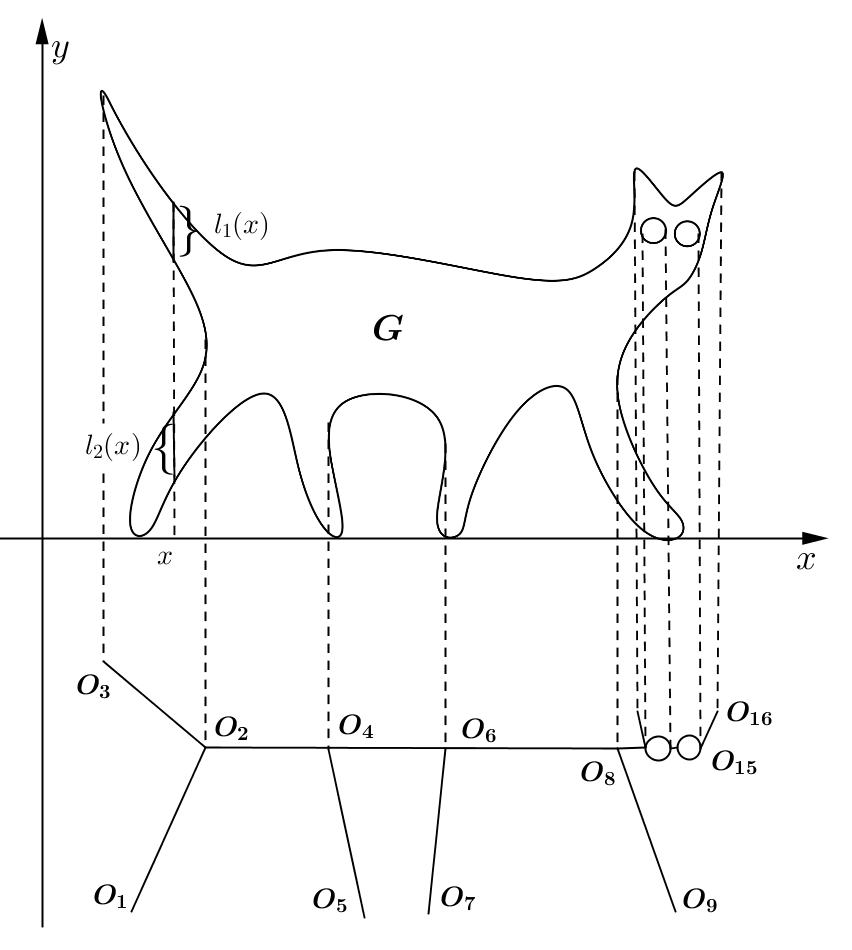}
\caption{}
\end{figure}

\medskip

In \cite{fw12} it has been proven that the process $\Pi^\e(t)$ converges, as $\e\downarrow 0$, to a continuous Markov process $\bar{Z}(t)$ on the graph $\Gamma$. More precisely, it has been proven that for any $z \in\,G$ and any bounded and continuous functional $F$ on $C([0,T];\Gamma)$, with $T>0$,  
\begin{equation}
\label{limiteintro}
\lim_{\e\to 0}\, \E_z F(\Pi^\e(\cdot))=\bar{\E}_{\Pi(z)}F(\bar{Z}(\cdot)).
\end{equation}
Notice that in \cite{fw12} the generator $\bar{L}$ of the Markov process $\bar{Z}(t)$  is explicitly described in terms of certain second order differential operators ${\cal L}_k$, acting in the interior of each edge $I_k$ of $\Gamma$, and of suitable gluing conditions, given at the vertices of $\Gamma$.  

Since in \eqref{variation-intro} the solution $u_\e$ of equation \eqref{intro1}, with $\partial w^Q/\partial t$ replaced by the regular function $h$, has been represented in terms of the process $(X^\e(t),Y^\e(t))$, one would like to be able to use  \eqref{limiteintro} to study the limiting behavior of $u_\e$ on $[0,+\infty)\times \Gamma$, as $\e\downarrow 0$. As a matter of fact, we have shown that for any $\varphi \in\,C(\bar{G})$, $z \in\,G$ and $0<\tau<T$, it holds
\begin{equation}
\label{cfintro1}
\lim_{\e\to 0}\sup_{t \in\,[\tau,T]}\le|\E_z\,\varphi(X^\e(t),Y^\e(t))-\bar{\E}_{\Pi(z)}\,\varphi^\wedge(\bar{Z}(t))\r|=0,\end{equation}
where
\[\varphi^\wedge (x,k)=\frac 1{l_k(x)}\int_{C_k(x)}\varphi(x,y)\,dy,\ \ \ \ (x,k) \in\,\Gamma,\]
and $l_k(x)$ is the length of the connected component $C_k(x)$ of the section $C(x)=\{(x,y) \in\,G\}$, corresponding to the edge $I_k$.

As a consequence of \eqref{cfintro1}, we have obtained that for any $\varphi \in\,C(\bar{G})$
\begin{equation}
\label{intro5}
\lim_{\e\to 0}\sup_{t \in\,[\tau,T]}|S_\e(t)\varphi-\bar{S}(t)^\vee\varphi|_{L^2(G)}=0,\end{equation} where
$\bar{S}(t)$ is the Markov transition semigroup associate with $\bar{L}$ and, for any $A \in\,\mathcal{L}(L^2(G))$,
\[A^\vee \varphi(x,y)=A\varphi^\wedge (\Pi(x,y)),\ \ \ (x,y) \in\,G.\]
In particular, due to \eqref{intro5}, we have shown that
\[\lim_{\e\to 0}\sup_{t \in\,[\tau,T]}|u_\e(t)-\bar{u}(t)\circ \Pi|_{L^2(G)}=0,\]
where $\bar{u}$ is the solution of the 
partial differential equation on $\Gamma$
\begin{equation}
\label{intro6}
\frac{\partial \bar{u}}{\partial t}(t,x,k)=\bar{L} \bar{u}(t,x,k)+h^\wedge (t,x,k),\ \ \ \ \bar{u}(0,x,k)=u_0^\wedge (x,k),
\end{equation}
endowed with suitable gluing conditions at the vertices of $\Gamma$.

We would like to stress the fact that \eqref{cfintro1} is not a straightforward consequence of \eqref{limiteintro}. Actually, \eqref{cfintro1} is a consequence of the following two limits
\begin{equation}
\label{cf6intro}
\lim_{\e\to 0}\sup_{t \in\,[\tau,T]}\le|\E_z\,\le[\varphi(Z^\e(t))-\varphi^\wedge(\Pi^\e(t))\r]\r|=0,
\end{equation}
and
\begin{equation}
\label{cf100intro}
\lim_{\e\to 0}\sup_{t \in\,[\tau,T]}\le|\mathbb{E}_z \varphi^\wedge(\Pi^\e(t))-\bar{{\mathbb E}}_{\Pi(z)}\varphi^\wedge(\bar{Z}(t))\r|=0.
\end{equation}

Limit \eqref{cf100intro} would be an immediate  consequence of \eqref{limiteintro}, if for any $\varphi \in\,C(\bar{G})$, the function $\varphi^\wedge$ were a continuous function on $\bar{\Gamma}$. Unfortunately, in general $\varphi^\wedge$ is not continuous at the internal vertices of $\Gamma$, so that the proof of \eqref{cf6intro} requires a thorough analysis, which also involves a few  estimates of the exit times of the process $Z_\e(t)$ from suitable small neighborhoods of the points $(x,y) \in\,\partial G$ where $\nu_2(x,y)=0$.

Concerning limit \eqref{cf6intro}, it follows from an  averaging argument, but its proof requires a suitable localization in time in the same spirit of Khasminski's paper \cite{khas}. Here the localization procedure is more delicate than in the classical setting considered by Khasminski, as it involves a stochastic differential equation with reflection and hence requires suitable estimates for the time increments of the local time of the process $(X^\e(t),Y^\e(t))$ at the boundary of $G$.

Now, once we have obtained \eqref{intro5}, we go back to the original problem, where equation \eqref{intro1} is perturbed by a cylindrical Wiener process $w^Q(t)$ given by
\[w^Q(t,x,y)=\sum_{k=1}^\infty (Qe_k)(x,y)\beta_k(t).\]
Here $Q$ is a bounded linear operator in $L^2(G)$, $\{e_k\}_{k \in\,\nat}$ is a complete orthonormal system in $L^2(G)$, and $\{\beta_k(t)\}_{k \in\,\nat}$ is a sequence of mutually independent Brownian motions. 
Under standard conditions on $b$ and  $Q$ (see e.g. \cite{DPZ} for all details), for any $\e>0$ equation \eqref{intro1} admits a unique mild solution $u_\e$. More precisely, there exists a unique adapted process $u_\e$ in $L^p(\Omega, C([0,T];L^2(G)))$, for any $T>0$ and $p\geq 1$, such that
\[u_\e(t)=S_\e(t)u_0+\int_0^t S_\e(t-s)b(u_\e(s))\,ds+\int_0^tS_\e(t-s)dw^Q(s).\]

In fact, here we assume the stronger condition that the covariance $QQ^\star$ of the cylindrical Wiener pricess $w^Q(t)$ is a trace class operator in $L^2(G)$, so that $w^Q(t) \in\,L^2(\Omega;L^2(G))$. In view of this, we have that the process
\[\bar{w}^Q(t):=w^Q(t)^\wedge=\sum_{k=1}^\infty (Qe_k)^\wedge \beta_k(t),\ \ \ \ t\geq 0,\]
is well defined in $L^2(\Omega;L^2(\Gamma,\nu))$, where $\nu$ is the invariant measure associated with the process $\bar{Z}(t)$. Thus, as $\bar{S}(t)$ is a contraction in $L^2(\Gamma,\nu)$, the process
\[\bar{w}_{\bar{L}}(t):=\int_0^t\bar{S}(t-s)d\bar{w}^Q(s),\ \ \ \ t\geq 0,\]
belongs to $L^p(\Omega;C([0,T];L^2(\Gamma,\nu)))$, for any $T>0$ and $p\geq 1$. In particular, this implies the following SPDE on the graph $\Gamma$
\begin{equation}
\label{intrograph}
d\bar{u}(t)=\le[\bar{L} \bar{u}(t)+b(\bar{u}(t))\r]\,dt+d\bar{w}^Q(t),\ \ \ \ \ \bar{u}(0)=u_0^\wedge,
\end{equation}
is well posed in $L^p(\Omega;C([0,T];L^2(\Gamma,\nu)))$.
We would like to stress that, as to our knowledge, this seems to be the first time a stochastic partial differential equation on a graph is studied.

Once we have obtained the well posedness of equation \eqref{intrograph}, we have shown that, as in the deterministic case,  $\bar{u}$ can be obtained as the limit of the solution $u_\e$ of equation \eqref{intro1}, as $\e\downarrow 0$. To this purpose, we have first shown that 
\[\lim_{\e\to 0}\mathbb{E}\sup_{t \in\,[0,T]}\le|\int_0^t S_\e(t-s)dw^Q(s)-\bar{w}_{\bar{L}}(t)\circ \Pi\r|_{L^2(G)}=0.\]
And, as $b$ is assumed to be Lipschitz-continuous, we have shown that this implies that
\[\lim_{\e\to 0}\mathbb{E}\sup_{t \in\,[\tau,T]}|u_\e(t)-\bar{u}(t)\circ \Pi|_{L^2(G)}=0.\]

\section{Some notations and a preliminary result}
\label{sec1}

In this section we introduce some notations and recall some important results from \cite{fw12}.

\subsection{The domain $G$, the narrow channel $G_\e$ and the graph $\Gamma$}

Let $G$ be a bounded open domain in $\reals^2$, having a  smooth boundary $\partial G$ (of class $C^3$), and for any $(x,y) \in\,\partial G$ let $\nu(x,y)$ denote the unit inward normal vector at the point $(x,y)$.  In what follows, we shall assume that $G$ satisfies the {\em uniform exterior sphere} condition, that is there exists $r_0>0$ such that for any $z \in\,\partial G$ there exists $z^\prime \in\,\reals^2$ with $|z-z^\prime|=r_0$ and $B(z^\prime,r_0)\cap G=\emptyset$. As a consequence of this assumption, there exists some constant $\kappa_0\geq 0$ such that for any $z \in\,\partial G$ and $z^\prime \in \bar{G}$
\begin{equation}
\label{boundarycond}
\langle z-z^\prime,\nu(z)\rangle-\kappa_0\,|z-z^\prime|^2\leq 0
\end{equation}
(for a proof see \cite{ls84}).

Now, for every $\e>0$ we introduce the {\em narrow channel} associated with $G$
\[G_\e:=\le\{\,(x,y) \in\,\reals^2\,:\, (x,\e^{-1}y) \in\,G\,\r\},\]
and we denote by $\nu^\e(x,y)$ the unit inward normal vector at the point $(x,y) \in\,\partial G_\e$. Notice that 
\begin{equation}
\label{normal}
\nu^\e(x,\e y)=c_{\e}(x,y)(\e\,\nu_1(x,y), \nu_2(x,y)),\ \ \ \ (x,y) \in\,\partial G,\end{equation}
for some function $c_\e:\partial G\to [1,+\infty)$, such that
\[\sup_{\substack{(x,y) \in\,\partial G\\ \e>0}}c_\e(x,y)=c<\infty.\]

\medskip

In what follows, we shall  assume that the region $G$ satisfies the following properties.
\begin{enumerate}
\item[I.] There are only finitely many   $x \in\,\mathbb{R}$ for which $\nu_2(x,y)=0$, for some  $(x,y) \in\,\partial G$.
\item[II.] For every $x \in\,\reals$, the cross-section
$C(x)=\le\{ (x,y) \in\,G\r\}$
consists of a finite union of intervals. Namely, when $C(x)\neq \emptyset$, there exist $N(x) \in\,\nat$ and  intervals $C_1(x),\ldots,C_{N(x)}(x)$ such that
\[ C(x)=\bigcup_{k=1}^{N(x)}C_k(x).\]
\item[III.] If $x \in\,\mathbb{R}$ is such that $\nu_2(x,y) \neq 0$, then for any $k=1,\ldots,N(x)$ we have
 \[l_k(x)=|C_k(x)|> 0.\]
\end{enumerate}

If  we identify the points of each connected component $C_k(x)$ of each cross section $C(x)$, we obtain a graph $\Gamma$, with a finite number of vertices $O_i$, corresponding to the connected components containing  points $(x,y) \in\,\partial G$ such that $\nu_2(x,y)=0$, and with a finite number of edges $I_k$, connecting the vertices. On our graph there are two different types of vertices, exterior ones, that are connected to only one edge of the graph, and interior ones, that are connected to two or more edges. See Fig. 1.

On the graph $\Gamma$ a distance can be introduced in the following way. If $y_1=(x_1,k)$ and $y_2=(x_2,k)$ belong to the same edge $I_k$, then $d(y_1,y_2)=|x_1-x_2|$. In the case $y_1$ and $y_2$ belong to different edges, then
\[d(y_1,y_2)=\min\,\le\{d(y_1,O_{i_1})+d(O_{i_1},O_{i_2})+\cdots+d(O_{i_n},y_2)\,\r\},\]
where the minimum is taken over all possible paths from $y_1$ to $y_2$, through every possible sequence of vertices $O_{i_1},\ldots,O_{i_n}$, connecting $y_1$ to $y_2$.

Now, any point $z$ on the graph $\Gamma$ can be uniquely identified by two coordinates, the horizontal coordinate $x$ and the integer $k$ which denotes the edge $I_k$ the point $z$ belongs to. Notice that if $z$ is one of the interior vertices $O_i$, this second coordinate may not be chosen in a unique way, as there are two or more edges having $O_i$ as their endpoint.

In what follows, we shall denote by $\Pi:G\to  \Gamma$ the identification map of the domain $G$ onto the corresponding graph $\Gamma$.
  For any vertex $O_i$ on the graph $\Gamma$, we denote by $E_i$ the set $\Pi^{-1}(O_i)$ consisting of points $(x,y) \in\,\partial G$ such that $\nu_2(x,y)=0$. The set $E_i$ can be one point, several points or an interval. 
In what follows, we shall assume that $G$ satisfies the following condition:
\begin{enumerate}
\item[IV.] For each vertex $O_i$, either $\nu_1(x,y)>0$, for all $(x,y) \in\,E_i$, or $\nu_1(x,y)<0$, for all $(x,y) \in\,E_i$.
\end{enumerate}

\subsection{A limiting result}

For each $\e>0$ and $z=(x,y) \in\,G$, we consider the stochastic system with reflecting boundary conditions on the domain $G$
\[\left\{\begin{array}{l}
\ds{dX^\e(t)=dB_1(t)+\nu_1(X^\e(t),Y^\e(t))\,d\phi^\e(t),\ \ \ \ X^\e(0)=x,}\\
\vs
\ds{dY^\e(t)=\frac 1\e dB_2(t)+\frac 1{\e^2}\nu_2(X^\e(t),Y^\e(t))\,d\phi^\e(t),\ \ \ \ Y^\e(0)=y.}
\end{array}\r.\]
Such a system can be rewritten as
\begin{equation}
\label{stochastic}
dZ^\e(t)= \sqrt{\si_\e}\,dB(t)+\si_\e\, \nu(Z^\e(t))\,d\phi^\e(t),\ \ \ \ Z^\e(0)=z \in\,G,
\end{equation}
where $\si_\e$ is the matrix defined by
\begin{equation}
\label{sigma}
\si_\e=\le(\ds{\begin{matrix}
1 &  0\\0  &\e^{-2}
\end{matrix}}\r).
\end{equation}
Here $B(t)$ is a  $2$-dimensional standard Brownian motion defined on some stochastic basis  $(\Omega, {\cal F}, \{{\cal F}_t\}_{t\geq 0}, \Pro)$
and $\phi^\e(t)$ is the local time of the process $Z^\e(t)$ on $\partial G$, that is the ${\cal F}_t$-adapted process, continuous with probability $1$, non-decreasing and increasing only when $Z^\e(t) \in\,\partial G$. More precisely, we have the following 
\begin{Definition}
The random pair $(Z^\e(t),\phi^\e(t))$, $t\geq 0$, is a solution of problem \eqref{stochastic} if $Z^\e(t)$ is a $\bar{G}$-valued $\{{\cal F}_t\}_{t\geq 0}$ semi-martingale and $\phi^\e(t)$ is a non-decreasing continuous process, such that
\[Z^\e(t)=z+\sqrt{\si_\e}\,B(t)+\int_0^t\si_\e\, \nu_1(Z^\e(s))\,d\phi^\e(s),\ \ \ \phi^\e(t)=\int_0^tI_{\{Z^\e(s) \in\,\partial G\}}\,d\phi^\e(s).\]

\end{Definition}

In \cite{fw12}, it has been studied the limiting behavior, as $\e\downarrow 0$, of the (non Markov) process 
$\Pi(Z^\e(t))$, $ t \geq 0$, in the space $C([0,T];\Gamma)$, for any fixed $T>0$ and $z \in\,G$.  Namely, it has been shown that the process $\Pi(Z^\e(t))$, which describes the slow motion of the process $Z^\e(t)$,  converges, in the sense of weak convergence of distributions in the space of continuous $\Gamma$-valued functions, to a diffusion process $\bar{Z}$  on $\Gamma$. 

The process $\bar{Z}$ has been described in terms of its generator $\bar{L}$, which is given by suitable differential operators $\bar{{\cal L}}_k$ within each edge $I_k=\{(x,k)\,:\,a_k\leq x\leq b_k\}$ of the graph and by certain gluing conditions at the vertices $O_i$ of the graph. More precisely,
for each $k$, the differential operator $\bar{{\cal L}}_k$ has the form
\begin{equation}
\label{lk}
\bar{{\cal L}}_k f(x)=\frac 1{2l_k(x)}\frac d{dx}\le(l_k\frac{df}{dx}\r)(x),\ \ \ \ a_k<x<b_k,
\end{equation}
and  the operator $\bar{L}$, acting on functions $f$ defined on the graph $\Gamma$, is defined as
\[\bar{L}f(x,k)=\bar{{\cal L}}_k f(x),\ \ \ \text{if}\ (x,k)\ \text{is an interior point of the edge $I_k$}.\]
The domain $D(\bar{L})$ is defined as the set of continuous functions on the graph $\Gamma$, that are twice continuously differentiable in the interior part of each edge of the graph, such that for any vertex $O_i=(x_i,k_1)=\cdots=(x_i,k_{N_i})$ there exist finite
\[ \lim_{(x,k_j)\to O_i}\bar{L} f(x,k_j),\]
the following one-sided limits exist
\[\lim_{x\to x_i} l_k(x)\frac{df}{dx}(x,k_j),\]
along any edge $I_{k_j}$ ending at the vertex $O_i=(x_i,k_j)$ and 
the following gluing condition is satisfied
\begin{equation}
\label{gluing}
\sum_{k=1}^{N_i}\le(\pm\lim_{x\to x_i}l_k(x)\frac{df}{dx}(x,k)\r)=0,\end{equation}
where the sign $+$ is taken for right limits and the sign $-$ for left limits.
In the case of an exterior vertex $O_i$, the gluing condition \eqref{gluing} reduces to
\begin{equation}
\label{gluingbis}
\lim_{x\to x_i}l_k(x)\frac{df}{dx}(x,k)=0,
\end{equation}
along the only edge $I_k$ terminating in $O_i$.

In \cite[Theorem 1.1]{fw12} it has been proven that for any domain $G$ satisfying properties I, II and III, there exists a continuous Markov process $\bar{Z}(t)$, $t\geq 0$, on the graph $\Gamma$ having $\bar{L}$ as its generator. In what follows we shall denote by $\bar{\Pro}_{(x,k)}$ and $\bar{\E}_{(x,k)}$ the probability and the expectation associated to the process $\bar{Z}(t)$, starting from the point $(x,k) \in\,\Gamma$. Moreover, we shall denote by $\bar{S}(t)$, $t\geq 0$, the transition semigroup associated with $\bar{Z}(t)$, defined by
\[\bar{S}(t)f(x,k)=\bar{\E}_{(x,k)}f(\bar{Z}(t)),\ \ \ \ t\geq 0,\ \ \ \ (x,k) \in\,\Gamma,\]
for any $f:\Gamma\to\reals$ Borel and bounded.

As we mentioned above, in \cite[Theorem 1.2]{fw12} it has also been proven that the process $\Pi(Z^\e)$ is weakly convergent to $\bar{Z}$ in $C([0,T];\Gamma)$, for any $T>0$ and $z \in\,G$. Namely, for any bounded and continuous functional $F$ on $C([0,T];\Gamma)$ and $z \in\,G$ it holds
\begin{equation}
\label{limite}
\lim_{\e\to 0}\, \E_z F(\Pi(Z^\e(\cdot)))=\bar{\E}_{\Pi(z)}F(\bar{Z}(\cdot)).\end{equation}

\section{Functions and operators on the graph $\Gamma$}
\label{sec2}

In what follows, for every $\e>0$  we denote $H_\e:=L^2(G_\e)$. In the special case $\e=1$, we  denote $H_1=:H$.
Moreover, we shall denote by $\bar{H}$ the space of measurable functions $f:\Gamma\to\reals$ such that
\[\sum_{k=1}^N\int_{I_k}|f(x,k)|^2l_k(x)\,dx<+\infty,\]
(here $N$ is the total number of edges in the graph $\Gamma$).
The space  $\bar{H}$ turns out to be a Hilbert space, endowed with the scalar product
\[\le<f,g\r>_{\bar{H}}=\sum_{k=1}^N\int_{I_k}f(x,k) g(x,k)l_k(x)\,dx.\]
Notice that, if we denote by $\nu$ the measure on $\Gamma$ defined by 
\begin{equation}
\label{cf1003}
\nu(A)=\sum_{k=1}^N\int_{I_k\cap A}f(x,k)l_k(x)\,dx,\ \ \ A \in\,\mathcal{B}(\Gamma),
\end{equation}
we have that $\bar{H}=L^2(\Gamma,\nu)$. 

Now, for any $u \in\,H$ we define
\begin{equation}
\label{bar}
u^{\wedge}(x,k)=\frac 1{l_k(x)}\int_{C_k(x)}u(x,y)\,dy,\ \ \ \ (x,k) \in\,\Gamma,\end{equation}
and for any $f \in\,\bar{H}$ we define
\[f^{\vee}(x,y)=f(\Pi(x,y)),\ \ \ (x,y) \in\,G.\]
For any $f \in\,\bar{H}$ and $u \in\,H$ we have
\begin{equation}
\label{cf50}
\langle u^\wedge, f\rangle_{\bar{H}}=\langle u,f^\vee\rangle_H,
\end{equation}
as
\[\langle u^\wedge, f\rangle_{\bar{H}}=\sum_{k=1}^\infty \int_{I_k}\frac{1}{l_k(x)}\int_{C_k(x)}u(x,y)\,dyf(x,k) l_k(x)\,dx=\sum_{k=1}^\infty \int_{I_k}\int_{C_k(x)}u(x,y) f^\vee(x,y)\,dy\,dx.\]
Moreover, for any $f \in\,\bar{H}$  we have
\begin{equation}
\label{cf46}
(f^{\vee})^\wedge=f.\end{equation}
Actually, for any $(x,k) \in\,\Gamma$ we have
\[\begin{array}{l}
\ds{(f^{\vee})^\wedge(x,k)=\frac 1{l_k(x)}\int_{C_k(x)}f^{\vee}(x,y)\,dy}\\
\vs
\ds{=\frac 1{l_k(x)}\int_{C_k(x)}f(\Pi(x,y))\,dy=\frac 1{l_k(x)}\int_{C_k(x)}f(x,k)\,dy=f(x,k).}
\end{array}\]
In particular, from \eqref{cf50} and \eqref{cf46}, we get that for any $f, g \in\,\bar{H}$
\begin{equation}
\label{cf2010}
\langle f^\vee,g^\vee\rangle_H=\langle (f^\vee)^\wedge,g\rangle_{\bar{H}}=\langle f,g\rangle_{\bar{H}}.
\end{equation}
This implies the following result. 
\begin{Lemma}
\label{l2}
If $\{f_n\}_{n \in\,\nat}$ is an orthonormal system in $\bar{H}$, then
the family of functions $\{f^{\vee}_n\}_{n \in\,\nat}$ is an orthonormal system  in $H$.
\end{Lemma}

Moreover, we have the following result.

\begin{Lemma}
\label{l1bis}
The mapping $f \in\,\bar{H}\mapsto f^{\vee} \in\,H$ is an isometry and the mapping
$u \in\,H\mapsto u^{\wedge} \in\,\bar{H}$ is a contraction.
\end{Lemma}

\begin{proof}
Due to \eqref{cf2010},  we have 
\[|f^\vee|_H^2=\langle f^\vee, f^\vee\rangle_H=\langle f, f\rangle_{\bar{H}}=|f|^2_{\bar{H}}.\]
Moreover, as a consequence of the H\"older inequality,  we have
\[\begin{array}{l}
\ds{|u^{\wedge}|_{\bar{H}}^2=\sum_{k=1}^N\int_{I_k}\le|\frac 1{l_k(x)}\int_{C_k(x)}u(x,y)\,dy\r|^2l_k(x)\,dx}\\
\vs
\ds{\leq \sum_{k=1}^N\int_{I_k}\frac 1{l_k(x)}\int_{C_k(x)}|u(x,y)|^2\,dyl_k(x)\,dx=\sum_{k=1}^N\int_{I_k}\int_{C_k(x)}|u(x,y)|^2\,dy\,dx=|u|^2_H.}
\end{array}\]

\end{proof}

\begin{Remark}{\em If $f \in\,C(\bar{\Gamma})$, then clearly $f^\vee \in\,C(\bar{G})$. On the other hand, if $\varphi \in\,C(\bar{G})$, it is not true, in general, that $\varphi^\wedge  \in\,C(\bar{\Gamma})$. Actually, $\varphi^\wedge$ may fail  to be continuous in correspondence of the interior vertices.}
\end{Remark}

Now, let $\{f_n\}_{n \in\,\nat}$ be a complete orthonormal system in $\bar{H}$. In what follows, we will denote by $K_1:=\langle f^{\vee}_n\rangle_{n \in\,\nat}$ and by $K_2:=K_1^\perp$, so that $H=K_1\oplus K_2$. This means that any $u \in\,H$ can be written as $u_1+u_2$, with $u_i \in\,K_i$, for $i=1,2$.
\begin{Lemma}
\label{l4}
We have
\begin{equation}
\label{cf51}
u \in\,K_2\Longleftrightarrow u^{\wedge}=0.\end{equation}
Moreover,
\begin{equation}
\label{cf52}
u \in\,K_1\Longleftrightarrow (u^{\wedge})^\vee=u.\end{equation}
\end{Lemma}
\begin{proof}
Thanks to \eqref{cf50},
for any $n \in\,\nat$ we have
\[\begin{array}{l}
\ds{\langle u^{\wedge},f_n\rangle_{\bar{H}} =\langle u,f^{\vee}_n\rangle_H.}
\end{array}\]
Since $\{f_n\}_{n  \in\,\nat}$ is a complete orthonormal system in $\bar{H}$, this implies \eqref{cf51}.

Next, if $u \in\,K_1$, then due to \eqref{cf46}
\[u^\wedge=\sum_{j=1}^\infty \langle u,f_j^\vee\rangle_H (f_j^\vee)^\wedge=\sum_{j=1}^\infty \langle u,f_j^\vee\rangle_H f_j.\]
Therefore, we get
\[(u^\wedge)^\vee=\sum_{j=1}^\infty \langle u,f_j^\vee\rangle_H f_j^\vee=u.\]
\end{proof}

\medskip

Now, for any $Q \in\,{\cal L}( H)$ and $f \in\,\bar{H}$, we define
\begin{equation}
\label{cf35}
Q^\wedge f=(Q f^\vee)^\wedge.\end{equation}
Due to Lemma \ref{l1bis}, it is immediate to check that $Q^\wedge \in\,{\cal L}(\bar{H})$ and
\[\|Q^\wedge \|_{{\cal L}(\bar{H})}\leq \|Q\|_{{\cal L}(H)}.\] Moreover, thanks again to Lemma \ref{l1bis}, if $\{f_n\}_{n \in\,\nat}$ is a complete orthonormal system in $\bar{H}$, we have
\[\sum_{n=1}^\infty|Q^\wedge f_n|_{\bar{H}}^2=\sum_{n=1}^\infty|(Q f_n^\vee)^\wedge|_{\bar{H}}^2\leq \sum_{n=1}^\infty|Q f_n^\vee|_{H}^2,\]
so that, 
 thanks to Lemma \ref{l2}, if $Q \in\,{\cal L}_2(H)$ we get $Q^\wedge \in\,{\cal L}_2(\bar{H})$ and
\[\|Q^\wedge \|_{{\cal L}_2(\bar{H})}\leq \|Q\|_{{\cal L}_2(H)}.\]
Thus, we have proven the following property.
\begin{Lemma}
\label{l2.3}
If $Q \in\,{\cal L}_2(H)$, then $Q^\wedge \in\, {\cal L}_2(\bar{H})$.
\end{Lemma}

Next, for any $A \in\,{\cal L}(\bar{H})$ and $u \in\,H$ we define
\begin{equation}
\label{cf36}
A^\vee u=(A u^\wedge )^\vee.\end{equation}
Due to Lemma \eqref{l1bis}, we have that $A^\vee \in\,{\cal L}(H)$ and 
\[\|A^\vee\|_{{\cal L}(H)}\leq \|A\|_{{\cal L}(\bar{H})}.\]

Moreover, we have
\begin{equation}
\label{cf47}
(A^\vee)^\wedge =A.
\end{equation}
Actually, due to \eqref{cf46}, for any $f \in\,\bar{H}$ we have
\[(A^\vee)^\wedge f=(A^\vee f^\vee)^\wedge =((A(f^\vee)^\wedge)^\vee)^\wedge=Af,\]
so that \eqref{cf47} follows.
Notice that  in general $(Q^\wedge)^\vee\neq Q$, for $Q \in\,{\cal L}(H)$, as in general $(u^\wedge)^\vee\neq u$, for $u \in\,H$. Actually, as a consequence of Lemma \ref{l4}, we have 
\[(Q^\wedge)^\vee= Q \Longleftrightarrow \text{ Ker}\, Q\subseteq K_2,\ \ \ \text{ Im}\, Q\subseteq K_1.\]

\section{An approximation result}
\label{app1}

We assume  here that the domain $G$ has the special form
\[G=\le\{(x,y) \in\,\reals^2\,:\,h_1(x)\leq y\leq h_2(x),\ x \in\,\reals\r\},\]
for some functions $h_1, h_2 \in\,C^3_b(\reals)$, such that 
\begin{equation}
\label{G}
h_2(x)-h_1(x)=:l(x)\geq l_0>0,\ \ \ \ x \in\,\reals.
\end{equation}
In this case we have
\[\partial G=\le\{(x,h_1(x))\,:\,x \in\,\reals\r\}\cup \le\{(x,h_2(x))\,:\,x \in\,\reals\r\},\]
and,  for any $x \in\,\reals$,
\[\nu(x,h_i(x))=(1+|h_i^\prime(x)|^2)^{-\frac 12}((-1)^i h_i^\prime(x),(-1)^{i+1}),\ \ \ \ i=1,2.\]
The corresponding graph $\Gamma$ consists of just one edge $I_1=\reals$ and the projected process $\Pi(Z^\e(t))$ is $(X^\e(t),1)$. Moreover, the limiting process $\bar{Z}(t)$, described in Section \ref{sec1},  is the solution of the stochastic equation
\[d\bar{Z}(t)=\frac 12 \frac{l^\prime(\bar{Z}(t))}{l(\bar{Z}(t))}\,dt+dB(t),\ \ \ \ \bar{Z}(0)=x.\]

\begin{Lemma}
There exists $\e_0>0$  such that for any $\e\leq \e_0$, $z \in\,G$  and $0\leq r<t$  
\begin{equation}
\label{phi}
\phi^\e(t)-\phi^\e({r})\leq c\,\e^2(1+(t-r))+\e^2\,\le|\int_{r}^tF^{\e,z}_1(s)\,dB_1(s)\r|+\e\le|\int_{r}^tF^{\e,z}_2(s)\,dB_2(s)\r|+c\,(t-r),
\end{equation}
where $F^{\e,z}_1(t)$ and $F^{\e,z}_2(t)$ are two adapted  processes  such that 
\[\sup_{\e>0,\ z \in\,G,\ t\geq 0}\le(|F^{\e,z}_1(t)|+|F^{\e,z}_1(t)|\r)=:M<\infty,\ \ \ \ \mathbb{P}-\text{a.s}.\]
\end{Lemma}

\begin{proof}
By proceeding as in \cite[Section 3]{fw12}, we denote by $u(x,y)$  the solution of the problem
\begin{equation}
\label{u}
\le\{\begin{array}{l}
\ds{\frac{\partial^2u}{\partial y^2}(x,y)=-\frac 1{l(x)}\le(\sqrt{1+|h_1^\prime(x)|^2}+\sqrt{1+|h_2^\prime(x)|^2}\r),\ \ \ \ (x,y) \in\,G,}\\
\vs
\ds{\frac{\partial u}{\partial y}(x,h_1(x))=\sqrt{1+|h_1^\prime(x)|^2},\ \ \ \ \frac{\partial u}{\partial y}(x,h_2(x))=-\sqrt{1+|h_2^\prime(x)|^2},}\\
\vs
\ds{u(x,h_1(x))=0,\ \ \ \ x \in\,\reals.}
\end{array}\r.\end{equation}
It is easy to compute explicitly $u$ and it turns out that
$u \in\,C^2_b(G)$. As a consequence of the It\^o formula, we have
\[\begin{array}{l}
\ds{u(Z^\e(t))-u(Z^\e(r))=\int_{r}^t\frac{\partial u}{\partial x}(Z^\e(s))dB_1(s)+\frac 1{\e}\int_{r}^t\frac{\partial u}{\partial y}(Z^\e(s))dB_2(s)}\\
\vs
\ds{+\int_{r}^t\le[\frac{\partial u}{\partial x}(Z^\e(s))\nu_1(Z^\e(s))+\frac 1{\e^2}\frac{\partial u}{\partial y}(Z^\e(s))\nu_2(Z^\e(s))\r]\,d\phi^\e(s)}\\
\vs
\ds{+\frac 12 \int_{r}^t\le[\frac{\partial^2 u}{\partial x^2}(Z^\e(s))+\frac 1{\e^2}\frac{\partial^2 u}{\partial y^2}(Z^\e(s))\r]\,ds,}
\end{array}\]
so that, thanks to \eqref{u}, we obtain
\[\begin{array}{l}
\ds{u(Z^\e(t))-u(Z^\e(r)=\int_{r}^t\frac{\partial u}{\partial x}(Z^\e(s))dB_1(s)+\frac 1{\e}\int_{r}^t\frac{\partial u}{\partial y}(Z^\e(s))dB_2(s)}\\
\vs
\ds{+\int_{r}^t\frac{\partial u}{\partial x}(Z^\e(s))\nu_1(Z^\e(s))\,d\phi^\e(s)+\frac 12 \int_{r}^t\frac{\partial^2 u}{\partial x^2}(Z^\e(s))\,ds-\frac 1{\e^2}\int_{r}^t\a(X^\e(s))\,ds}\\
\vs
\ds{+\frac 1{\e^2}\le(\phi^\e(t)-\phi^\e({r})\r),}
\end{array}\]
where
\[\a(x)=\frac 1{l(x)}\le(\sqrt{1+|h_1^\prime(x)|^2}+\sqrt{1+|h_2^\prime(x)|^2}\r).\]
This implies
\[\begin{array}{l}
\ds{\phi^\e(t)-\phi^\e({r})=\e^2\le(u(Z^\e(t))-u(Z^\e({r}))\r)}\\
\vs
\ds{-\e^2\int_{r}^t\frac{\partial u}{\partial x}(Z^\e(s))dB_1(s)-\e\int_{r}^t\frac{\partial u}{\partial y}(Z^\e(s))dB_2(s)}\\
\vs
\ds{-\e^2 \int_{r}^t\frac{\partial u}{\partial x}(Z^\e(s))\nu_1(Z^\e(s))\,d\phi^\e(s)-\frac {\e^2}2 \int_{r}^t\frac{\partial^2 u}{\partial x^2}(Z^\e(s))\,ds+\int_{r}^t\a(X^\e(s))\,ds,}
\end{array}\]
and then
\[\begin{array}{l}
\ds{\phi^\e(t)-\phi^\e({r})\leq c\,\e^2(1+(t-r))+\e^2\,\le|\int_{r}^t\frac{\partial u}{\partial x}(Z^\e(s))dB_1(s)\r|+\e\le|\int_{r}^t\frac{\partial u}{\partial y}(Z^\e(s))dB_2(s)\r|}\\
\vs
\ds{+c\,\e^2(\phi^\e(t)-\phi^\e({r}))+\int_{r}^t\a(X^\e(s))\,ds.}
\end{array}\]
In particular, if we take $\e_0=1/\sqrt{c 2}$, we can conclude that
\[\begin{array}{l}
\ds{\phi^\e(t)-\phi^\e({r})\leq c\,\e^2(1+(t-r))+c\,(t-r)}\\
\vs
\ds{+\e^2\,\le|\int_{r}^t\frac{\partial u}{\partial x}(Z^\e(s))dB_1(s)\r|+\e\le|\int_{r}^t\frac{\partial u}{\partial y}(Z^\e(s))dB_2(s)\r|,\ \ \ \ \e\leq \e_0,}
\end{array}
\]
and this yields \eqref{phi}.

\end{proof}

Now, for any $\e, \gamma>0$, we consider the stochastic Skorokhod problem 
\begin{equation}
\label{discrete}
\le\{\begin{array}{l}
\ds{d{Z}^{\e,\gamma}(t)=\sqrt{\hat{\si}_\e}\, dB(t)+\hat{\si}_\e\,\nu({Z}^{\e,\gamma}(t))\,d{\phi}^{\e,\gamma}(t),\ \ \ \ t \in\,[k\gamma,(k+1)\gamma),}\\
\vs
\ds{{Z}^{\e,\gamma}(k\gamma)=Z^\e(k\gamma),}
\end{array}\r.
\end{equation}
where 
\begin{equation}
\label{sigmahat}
\hat{\si}_\e=\le(\ds{\begin{matrix}
0 &  0\\0  &\e^{-2}
\end{matrix}}\r).
\end{equation}
Clearly, for any $t \in\,[k\gamma,(k+1)\gamma)$ the variable ${Z}^{\e,\gamma}(t)$ lives in   the random interval 
\[C(X^\e(k\gamma))=[h_1(X^\e(k\gamma)),h_2(X^\e(k\gamma))].\] Moreover, for any $t \in\,[k\gamma,(k+1)\gamma)$ we have that
${Z}^{\e,\gamma}(t)=(X^{\e}(k\gamma),{Y}^{\e,\gamma}(t))$, where ${Y}^{\e,\gamma}$ solves the problem
\[\le\{
\begin{array}{l}
\ds{d{Y}^{\e,\gamma}(t)=\frac 1{\e}\,dB_2(t)+\frac 1{\e^2}\,\nu_2(X_z^\e(k\gamma),{Y}^{\e,\gamma}(t))\,d{\phi}^{\e,\gamma}(t),\ \ \ \ t \in\,[k\gamma,(k+1)\gamma),}\\
\vs
\ds{{Y}^{\e,\gamma}(k\gamma)=Y_z^{\e}(k\gamma).}
\end{array}\r.\]

\begin{Lemma}
\label{lemmaA.2}
For any $p\geq 1$ there exists $c_p>0$ such that for any $\e,\gamma>0$, $k \in\,\nat$ and $t,s \in\,[k\gamma,(k+1)\gamma)$
\begin{equation}
\label{phihat}
\sup_{z \in\,G}\E_z\,|{\phi}^{\e,\gamma}(t)-{\phi}^{\e,\gamma}(s)|^p\leq c_p\le(\gamma^{p}+\e^p\gamma^{p/2}+\e^{2p}\r).
\end{equation}

\end{Lemma}

\begin{proof}
We have that
\begin{equation}
\label{cf11}
({Z}^{\e,\gamma}(t+k\gamma),{\phi}^{\e,\gamma}(t+k\gamma))\sim (Z^\e_1(t),\phi^\e_1(t)),\ \ \ \ \ t \in\,[0,\gamma),\end{equation}
where
\begin{equation}
\label{cf12}
\le\{
\begin{array}{l}
\ds{dZ^\e_1(t)=\sqrt{\hat{\si}_\e}\,dB^1(t)+\hat{\si}_\e\,\nu_2(Z^\e_1(t))\,d\phi^\e_1(t),\ \ \ \ t \in\,[0,\gamma),}\\
\vs
\ds{Z^\e_1(0)=Z^{\e}(k\gamma),}
\end{array}\r.\end{equation}
for some $2$-dimensional Brownian motion $B^1(t)$  such that $Z^\e(k\gamma)$  is independent of $B^1(t)$, for $t\geq 0$. Moreover, we have
\[(Z^\e_1(t),\phi^\e_1(t))\sim (Z_2(t/\e^2),\e^2\phi_2(t/\e^2)),\]
where 
\begin{equation}
\label{cf13}
\le\{
\begin{array}{l}
\ds{dZ_2(t)=\sqrt{\hat{\si}_1}d\tilde{B}(t)+\hat{\si}_1\nu_2(Z_2(t))\,d\phi_2(t),\ \ \ \ t \in\,[0,\gamma/\e^2),}\\
\vs
\ds{Z_2(0)=Z^{\e}(k\gamma),}
\end{array}\r.\end{equation}
for some Brownian motion $\tilde{B}(t)$ such that $Z^\e(k\gamma)$  is independent of $\tilde{B}(t)$, for $t\geq 0$.
In particular, this implies that for any $t, s \in\,[k\gamma,(k+1)\gamma)$
\begin{equation}
\label{phihat1}
\E_z\,|{\phi}^{\e,\gamma}(t)-{\phi}^{\e,\gamma}(s)|^p=\e^{2p}\,\tilde{\E}\,|\phi_2((t-k\gamma)/\e^2)-\phi_2((s-k\gamma)/\e^2)|^p.\end{equation}

Now, if $u$ is the same function introduced in \eqref{u}, from It\^o's formula we have
\[\begin{array}{l}
\ds{u(Z_2(t))-u(Z_2(s))=\frac 12\int_s^t \frac{\partial^2 u}{\partial y^2}(Z_2(r))\,dr}\\
\vs
\ds{+
\int_s^t\frac{\partial u}{\partial y}(Z_2(r))\nu_2(Z_2(r))d\phi_2(r)+\int_s^t \frac{\partial u}{\partial y}(Z_2(r))\,d\tilde{B}(r)}\\
\vs
\ds{=-\int_s^t \a(X^\e(k\gamma))\,dr +\int_s^t \frac{\partial u}{\partial y}(Z_2(r))\,d\tilde{B}(r)+\phi_2(t)-\phi_2(s).}
\end{array}\]
This implies
\[\begin{array}{l}
\ds{\tilde{\E}|\phi_2(t)-\phi_2(s)|^p\leq c_p\,(|t-s|^p+|t-s|^{p/2}+1),}
\end{array}\]
so that, thanks to \eqref{phihat1}, we can conclude
\[\E_z|{\phi}^{\e,\gamma}(t)-{\phi}^{\e,\gamma}(s)|^p\leq c_p\,\e^{2p}\le(\frac{|t-s|^p}{\e^{2p}}+\frac{|t-s|^{p/2}}{\e^{p}}+1\r)\leq c_p\le(\gamma^{p}+\e^p\gamma^{p/2}+\e^{2p}\r).\]

\end{proof}

Now, we can prove the main result of this section.
\begin{Theorem}
\label{lim-z-eps}
Assume that
\[G=\le\{(x,y) \in\,\reals^2\,:\,h_1(x)\leq y\leq h_2(x),\ x \in\,\reals\r\},\]
for some $h_1, h_2 \in\,C^3_b(\reals)$, such that 
\[\inf_{x \in\,\reals}h_2(x)-h_1(x):=l_0>0.\]
Then, there exists $\kappa_1>0$ such that, if we set
$\gamma_\e=\e^2\,\log \e^{-\kappa_1}$,
for any  $T>0$ it holds
\begin{equation}
\label{ave}
\lim_{\e\to 0}\, \sup_{z \in\,G}\,\sup_{t \in\,[0,T]}\,\E_z\,|Z^\e(t)-{Z}^{\e,\gamma_\e}(t)|^2=0.
\end{equation}
\end{Theorem}

\begin{proof} In what follows, for the sake of simplicity, we shall denote 
\[\hat{Y}^\e(t):={Y}^{\e,\gamma_\e}(t),\ \ \ \hat{Z}^\e(t):={Z}^{\e,\gamma_\e}(t),\ \ \ \ \hat{\phi}^\e(t):={\phi}^{\e,\gamma_\e}(t),\ \ \ \ \e>0.\]

As we are assuming the domain $G$ to be a smooth and   bounded open sets of $\reals^2$, by proceeding as in \cite{ls84} we can introduce an extension $\Psi \in\, C^1_b(\reals^2)$ of the distance function $d(\cdot,{\partial G})$, which is defined on the restriction to $G$ of a neighborhood of $\partial G$, such that
\begin{equation}
\label{cf70}
\nabla\Psi(x,y)=\nu(x,y),\ \ \ \ (x,y) \in\,\partial G.\end{equation}
Then, for each $\e>0$ we define
\[H_\e(t):=\exp\le(-\frac 1\a\le[\Psi(Z^\e(t))+\Psi(\hat{Z}^{\e}(t))\r]\r),\ \ \ \ t\geq 0,\]
where $\a>0$ is some constant to be chosen later. Notice that, as $\Psi$  is bounded, for any $\a>0$ 
there exists $c_\a>0$ such that
\begin{equation}
\label{psi}
c_\a<H_\e(t)<\frac 1{c_\a},\ \ \ \ t\geq 0.
\end{equation}
It is immediate to check that
\[\le\{
\begin{array}{l}
\ds{
d(Z^\e-\hat{Z}^{\e})(t)=\si\,dB(t)+\si_\e\,\nu(Z^\e(t))\,d\phi^\e(t)-\hat{\si}_\e\,\nu(\hat{Z}^{\e}(t))d\hat{\phi}^{\e}(t),\ \ \ \ t \in\,[k\gamma_\e,(k+1)\gamma_\e),}\\
\vs
\ds{(Z^\e-\hat{Z}^{\e,\gamma})(k\gamma_\e)=0,}
\end{array}\r.
\]
where 
\[\si=\sqrt{\si_\e}-\sqrt{\hat{\si}_\e}=\le(\ds{\begin{matrix}
1 &  0\\0  &0
\end{matrix}}\r).\]
Then, if we set $\Delta Y^{\e}(t)=Y^\e(t)-\hat{Y}^{\e}(t)$, for $t\geq 0$, thanks to  \eqref{stochastic}, \eqref{discrete} and \eqref{cf70}, as a consequence of It\^o's formula
 we obtain that for any $k \in\,\nat$ and $t \in\,[k\gamma_\e,(k+1)\gamma_\e)$
\begin{equation}
\label{e1}
\begin{array}{l}
\ds{H_\e(t)|\Delta Y^{\e}(t)|^2=\frac{2}{\e^2}\int_{k\gamma_\e}^tH_\e(s)\le[\Delta Y^{\e}(s)\,\nu_2(Z^\e(s))\,d\phi^\e(s)- \Delta Y^{\e}(s)\,\nu_2(\hat{Z}^{\e}(s))\,d\hat{\phi}^{\e}(s)\r]\,ds}\\
\vs
\ds{-\frac 1\a \int_{k\gamma_\e}^tH_\e(s)|\Delta Y^{\e}(s)|^2\le[\langle \nabla \Psi(Z^\e(s)),\sqrt{\si_\e} dB(s)\rangle+\langle \nabla \Psi(\hat{Z}^{\e}(s)),\sqrt{\hat{\si}_\e} dB(s)\rangle\r.}\\
\vs
\ds{+\langle \nabla \Psi(Z^\e(s)),\si_\e \nu(Z^\e(s)) \rangle\,d\phi^\e(s)+\langle \nabla \Psi(\hat{Z}^{\e}(s)),\hat{\si}_\e \nu(\hat{Z}^{\e}(s))\rangle \,d\hat{\phi}^{\e}(s)}\\
\vs
\ds{\le.+\frac 12 \text{Tr}\le[D^2\Psi(Z^\e(s))\si_\e\r]\,ds+\frac 12\text{Tr}\le[D^2\Psi(\hat{Z}^{\e}(s))\hat{\si_\e}\r]\,ds\r]}\\
\vs
\ds{+\frac 1{\a^2}\int_{k\gamma_\e}^tH_\e(s)|\Delta Y^{\e}(s)|^2\frac 1{\e^2} |\nu_1(Z^\e(s))+\nu_1(\hat{Z}^{\e}(s))|^2\,|\nu_2(Z^\e(s))+\nu_2(\hat{Z}^{\e}(s))|^2\,ds.}
\end{array}\end{equation}
Now, we have
\[\begin{array}{l}
\ds{\frac{1}{\e^2}\Delta Y^{\e}(s)\,\nu_2(Z^\e(s))-\frac 1{2\a}|\Delta Y^{\e}(s)|^2\langle \nabla \Psi(Z^\e(s)),\si_\e \nu(Z^\e(s)) \rangle}\\
\vs
\ds{=\frac{1}{\e^2}\langle Z^\e(s)-\hat{Z}^{\e}(s),\nu(Z^\e(s))\rangle-\frac{1}{\e^2}(X^\e(s)-\hat{X}^{\e}(s))\nu_1(Z^\e(s))}\\
\vs
\ds{-\frac 1{2\a}\le(|Z^\e(s)-\hat{Z}^{\e}(s)|^2-|X^\e(s)-\hat{X}^{\e}(s)|^2\r)\le(\nu_1^2(s)+\frac 1{\e^2}\,\nu_2^2(s)\r)}\\
\vs
\ds{\leq \frac{1}{\e^2}\le(\langle Z^\e(s)-\hat{Z}^{\e}(s),\nu(Z^\e(s))\rangle-\frac 1{2\a}|Z^\e(s)-Z^{\e}(s)|^2\r)}\\
\vs
\ds{-\frac{1}{\e^2}(X^\e(s)-\hat{X}^{\e}(s))\nu_1(Z^\e(s))+\frac c{2\a\e^2}|X^\e(s)-\hat{X}^{\e}(s)|^2,}
\end{array}\]
last inequality following from the fact that for any $z \in\,G$ and $\e \in\,(0,1)$
\[\frac c{\e^2}:=\frac 1{\e^2}\le(1+|h^\prime_2|_\infty^2\r)^{-1}\leq \nu_1^2(z)+\frac 1{\e^2}\nu_2^2(z)\leq \frac 1{\e^2}.\]
Then, thanks to \eqref{boundarycond}, there exists $\a>0$ such that  we have
\begin{equation}
\label{e2}
\begin{array}{l}
\ds{\frac{1}{\e^2}\Delta Y^{\e}(s)\,\nu_2(Z^\e(s))\,d\phi^\e(s)-\frac 1{2\a}|\Delta Y^{\e}(s)|^2\langle \nabla \Psi(Z^\e(s)),\si_\e \nu(Z^\e(s)) \rangle \,d\phi^\e(s)}\\
\vs
\ds{\leq \frac{1}{\e^2}\le(\frac 1{2\a}|X^\e(s)-\hat{X}^{\e}(s)|^2-(X^\e(s)-\hat{X}^{\e}(s))\nu_1(Z^\e(s))\r)
\,d\phi^\e(s).}
\end{array}\end{equation}
In the same way, we have
\begin{equation}
\label{e3}
\begin{array}{l}
\ds{-\frac 1{\e^2}\, \Delta Y^{\e}(s)\,\nu_2(\hat{Z}^{\e}(s))\,d\hat{\phi}^{\e}(s)-\frac 1{2\a} \,|\Delta Y^{\e}(s)|^2\langle \nabla \Psi(\hat{Z}^{\e}(s)),\hat{\si}_\e \nu(\hat{Z}^{\e}(s))\rangle \,d\hat{\phi}^{\e}(s)}\\
\vs
\ds{\leq \frac{1}{\e^2}\le(\frac 1{2\a}|X^\e(s)-\hat{X}^{\e}(s)|^2+(X^\e(s)-\hat{X}^{\e}(s))\nu_1(\hat{Z}^{\e}(s))\r)\,d\hat{\phi}^{\e}(s).}
\end{array}\end{equation}
Thus, if we use \eqref{e2} and \eqref{e3} in \eqref{e1}, thanks to \eqref{psi} we get
\[\begin{array}{l}
\ds{|\Delta Y^{\e}(t)|^2\leq \frac c{\e^2}\,\int_{k\gamma_\e}^t\le(|X^\e(s)-\hat{X}^{\e}(s)|^2+|X^\e(s)-\hat{X}^{\e}(s)|\r)\le(d\phi^\e(s)+d\hat{\phi}^{\e}(s)\r)}\\
\vs
\ds{+\le|\int_{k\gamma_\e}^t H_\e(s)|\Delta Y^{\e}(s)|^2\le[\langle \nabla \Psi(Z^\e(s)),\sqrt{\si_\e} dB(s)\rangle+\langle \nabla \Psi(\hat{Z}^{\e}(s)),\sqrt{\hat{\si}_\e} dB(s)\rangle\r] \r|}\\
\vs
\ds{+\frac c{\e^2}\int_{k\gamma_\e}^t|\Delta Y^{\e}(s)|^2\,ds,}
\end{array}\]
so that
\[\begin{array}{l}
\ds{|\Delta Y^{\e}(t)|^2\leq \frac c{\e^2}\int_{k\gamma_\e}^t|\Delta Y^{\e}(s)|^2\,ds+ \frac c{\e^2}\sup_{s \in\,[k\gamma_\e,(k+1)\gamma_\e]}\le(|X^\e(s)-\hat{X}^{\e}(s)|^2+|X^\e(s)-\hat{X}^{\e}(s)|\r)}\\
\vs
\ds{\times \le[\le(\phi^\e((k+1)\gamma_\e)-\phi^\e({k\gamma_\e})\r)+\le(\hat{\phi}^{\e}((k+1)\gamma_\e)-\hat{\phi}^{\e}({k\gamma_\e})\r)\r]}\\
\vs
\ds{+\le|\int_{k\gamma_\e}^t H_\e(s)|\Delta Y^{\e}(s)|^2\le[\langle \nabla \Psi(Z^\e(s)),\sqrt{\si_\e} dB(s)\rangle+\langle \nabla \Psi(\hat{Z}^{\e}(s)),\sqrt{\hat{\si}_\e} dB(s)\rangle\r] \r|.}
\end{array}\]
This implies,
\[\begin{array}{l}
\ds{\E_z\,|\Delta Y^{\e}(t)|^4\leq \frac{c}{\e^4}\le(\gamma_\e+\e^2\r)\int_{k\gamma_\e}^t\E_z\,|\Delta Y^{\e}(s)|^4\,ds}\\
\vs
\ds{+\frac c{\e^4}\,\Lambda_\e\le(\E_z\,\le|\phi^\e((k+1)\gamma_\e)-\phi^\e(k\gamma_\e)\r|^2+\E_z\,\le|\hat{\phi}^{\e}((k+1)\gamma_\e)-\hat{\phi}^{\e}(k\gamma_\e)\r|^2\r),}
\end{array}\]
where
\[\Lambda_\e:=\E_z\sup_{s \in\,[k\gamma_\e,(k+1)\gamma_\e]}\le(|X^\e(s)-\hat{X}^{\e}(s)|^4+|X^\e(s)-\hat{X}^{\e}(s)|^2\r).\]
Therefore, thanks to \eqref{phi} and \eqref{phihat}, we get
\[\begin{array}{l}
\ds{\E_z\,|\Delta Y^{\e}(t)|^4\leq \frac{c}{\e^4}\le(\gamma_\e+\e^2\r)\int_{k\gamma_\e}^t\E_z\,|\Delta Y^{\e}(s)|^4\,ds+\Lambda_\e\,\big[1+\le(\frac{\gamma_\e}{\e^2}\r)^2\big],}
\end{array}
\]
and, since $\e^2/\gamma_\e\leq c$, the Gronwall lemma gives
\begin{equation}
\label{cf1}
\E_z\,|\Delta Y^{\e}(t)|^4\leq c\,\Lambda_\e \big[1+\le(\frac{\gamma_\e}{\e^2}\r)^2\big]\exp\big[c\le(\frac{\gamma_\e}{\e^2}\r)^2\big].\end{equation}

Now,  for any $s \in\,[k\gamma_\e,(k+1)\gamma_\e]$ we have
\[X^\e(s)-\hat{X}^{\e}(s)=B_1(s)-B_1(k\gamma_\e)+\int_{k\gamma_\e}^s\nu_1(X^\e(r),Y^\e(r))\,d\phi^\e(r),\]
so that, thanks to \eqref{phi}, for any $p\geq 2$
\[\begin{array}{l}
\ds{\E_z\sup_{s \in\,[k\gamma_\e,(k+1)\gamma_\e]}|X^\e(s)-\hat{X}^{\e}(s)|^p\leq c_p \gamma_\e^{p/2}+c_p\, \E_z\,|\phi^\e((k+1)\gamma_\e)-\phi^\e(k\gamma_\e)|^p\leq c_p\,\gamma_\e^{p/2}.}\end{array}\]
This implies
$\Lambda_\e\leq c\,\gamma_\e,$
so that from \eqref{cf1} we get
\[\begin{array}{l}
\ds{\E_z\,|\Delta Y^{\e}(t)|^4\leq c\,\gamma_\e \big[1+\le(\frac{\gamma_\e}{\e^2}\r)^2\big]\exp\big[c\le(\frac{\gamma_\e}{\e^2}\r)^2\big]}\\
\vs
\ds{=c\,\e^2\,\log \e^{-\kappa_1}
\le(1+(\log \e^{-\kappa_1})^2\r)\exp\le(c\,\log \e^{-\kappa_1}\r).}
\end{array}\]
Therefore, if we take $\kappa_1<c/2$, we can conclude that  \eqref{ave} holds true.

\end{proof}

\section{The Neumann problem associated with the operator $L_\e$}
\label{sec4}

For any $\e>0$, we define 
\begin{equation}
\label{cf303}
{\cal L}_\e u(x,y)=\frac 12 \le(\frac{\partial^2 u}{\partial x^2}+\frac 1{\e^2}\frac{\partial^2 u}{\partial y^2}\r)(x,y)=\frac 12\,\text{div}\le(\si_\e \nabla u\r)(x,y),\ \ \ \ (x,y) \in\,G,\end{equation}
For any $\e>0$,  the uniformly elliptic second order differential operator ${\cal L}_\e$, endowed with the co-normal  derivative boundary condition 
\[\nabla u\cdot \si_\e \nu_{|_{\partial G}}=0,\]
generates a strongly continuous analytic semigroup $S_\e(t)$, $t\geq 0$, in the Hilbert space $H$ and in the Banach space $C(\bar{G})$. The generator of $S_\e(t)$ will be denoted by $L_\e$. For a proof of all these results see e.g. \cite{lunardi}.
Moreover, the Lebesgue measure on $G$ is  invariant  for the semigroup $S_\e(t)$, so that $S_\e(t)$ is a contraction on $H$ and $C(G)$.

In the present section we consider the  Cauchy linear problem associated with ${\cal L}_\e$
\begin{equation}
\label{neumann}
\le\{\begin{array}{l}
\ds{\frac{\partial \rho_\e}{\partial t}(t,x,y)={\cal L}_\e \rho_\e(t,x,y),\ \ \ \ \ (x,y) \in\,G,\ \ \ \ t>0,}\\
\vs
\ds{\nabla \rho_\e(t,x,y)\cdot \si_\e \nu(x,y)=0,\ \ \ (x,y) \in\,\partial G,\ \ \ \ \ \ \ \rho_\e(0,x,y)=\varphi(x,y),\ \ \ \ (x,y) \in\,G.}
\end{array}\r.
\end{equation}

It is well known (for a proof see e.g.  \cite[Theorem 2.5.1]{fred})  that the solution $\rho_\e(t)$ to  problem \eqref{neumann} has a probabilistic representation in terms of the solution  of the stochastic equation with  reflection \eqref{stochastic}.
Namely, it holds
\[\rho_\e(t,z)=S_\e(t)\varphi(z)=\E_z\, \varphi(Z^\e(t)),\ \ \ t\geq 0,\ \ \ z=(x,y) \in\,G.\]
Our aim here  is  studying the limiting behavior of $\rho_\e(t)$, as $\e\downarrow 0$.  

To this purpose, we first introduce some notation (see \cite{fw12} for all details). For any edge $I_k=\{(x,k)\,:\,a_k\leq x\leq b_k\}$ on the graph $\Gamma$ and for any $a_k\leq a<b\leq b_k$, we denote
\[G_k(a,b):=\{(x,y) \in\,\Pi^{-1}(I_k)\,:\,a<x<b\},\ \ \ G_k[a,b]:=\{(x,y) \in\,\Pi^{-1}(I_k)\,:\,a\leq x\leq b\},\]
and for any $\d>0$ we define
\[G(\d):=\bigcup_{k=1}^N G_k[a_k+\d,b_k-\d].\]
For any vertex $O_i=(x_i,k_1)=\ldots=(x_i,k_{s_i})$ and $a<x_i<b$ we denote
\[G(O_i,a,b):=\bigcup_{j=1}^{s_i}\le\{(x,y) \in\,\Pi^{-1}(I_{k_j})\,:\,x \in\,(a,b)\r\}.\]
Finally, for any vertex $O_i$ and edge $I_k$, having $O_i=(x_i,k)$ as one of its endpoints, and for any $\d>0$ we denote
\[C_{ik}(\d):=\{(x,y) \in\,\Pi^{-1}(I_k)\,:\,x=x_i\pm \d\},\]
and then we set
\[C(\d):=\bigcup_{i,k}C_{ik}(\d).\] 
Notice that if $0<\d^\prime<\d$, then
\[C_{ik}(\d^\prime)\subset G(O_i,x_i-\d,x_i+\d).\]

Next,  for any $\e, \d,\d^\prime>0$, with $0<\d^\prime<\d$, we introduce the following sequence of stopping times 
\[\si^{\e,\d,\d^\prime}_n=\min\{\,t\geq \tau^{\e,\d,\d^\prime}_n\,:\,Z^\e(t) \in\,G(\d)\,\},\ \ \ \ \ \tau^{\e,\d,\d^\prime}_n=\min\{t> \si^{\e,\d,\d^\prime}_{n-1}\,:\,Z^\e(t) \in\,C(\d^\prime)\},\]
with $\tau^{\e,\d,\d^\prime}_0=0$. 
For any fixed $\e>0$ we have that
\[\lim_{n\to \infty} \tau^{\e,\d,\d^\prime}_n=\lim_{n\to \infty} \si^{\e,\d,\d^\prime}_n=\infty,\ \ \ \ \mathbb{P}-\text{a.s.}\]
and for any $n \in\,\nat$
\[Z^\e(\tau^{\e,\d,\d^\prime}_n) \in\,C(\d^\prime),\ \ \ \ Z^\e(\si^{\e,\d,\d^\prime}_n) \in\,C(\d).\]
Moreover, if $Z^\e(0) \in\,G(\d)$, we have that $\si^{\e,\d,\d^\prime}_0=0$ and $\tau^{\e,\d,\d^\prime}_1$ is the first time the process $Z^\e(t)$ touches $C(\d^\prime)$.

\begin{Lemma}
\label{lemf}
If $G$ satisfies assumptions I-IV, then, for any $0<\tau<T$ and for any $\varphi \in\,C(\bar{G})$ and $z \in\,G$
\begin{equation}
\label{cf100}
\lim_{\e\to 0}\,\sup_{t \in\,[\tau,T]}\le|\mathbb{E}_z(\varphi^\wedge)^\vee(Z^\e(t))-\bar{{\mathbb E}}_{\Pi(z)}\varphi^\wedge(\bar{Z}(t))\r|=0.
\end{equation}
\end{Lemma}

\begin{proof}
As a consequence of limit \eqref{limite} (whose proof can be found in \cite[Theorem 1.2]{fw12}) and of the Skorokhod embedding theorem, we have that for any $\psi \in\,C(\bar{\Gamma})$
\begin{equation}
\label{cf1000}
\lim_{\e\to 0}\,\sup_{t \in\,[0,T]}\,\le| \mathbb{E}_z\, \psi^\vee(Z^\e(t)-\bar{{\mathbb E}}_{\Pi(z)}\psi(\bar{Z}(t))\r|=0.
\end{equation}
Thus, if $\varphi^\wedge$ were continuous on $\bar{\Gamma}$, then \eqref{cf100} would follow from  \eqref{cf1000} . Unfortunately, if $\varphi \in\,C(\bar{G})$ in general $\varphi^\wedge$ is not continuous on $\bar{\Gamma}$, so that we cannot use \eqref{cf1000} directly and we have to use an approximation argument.

If $\varphi \in\,C(\bar{G})$, it is immediate to check that  $\varphi^\wedge$ is everywhere continuous but at the interior vertices of the graph $\Gamma$. However, for any $\d>0$ there exists $\psi_\d \in\,C(\bar{\Gamma})$ such that
\[\|\psi_\d\|_\infty\leq \|\varphi^\wedge\|_\infty,\ \ \ \ \ \ \psi_\delta\equiv \varphi^\wedge\ \ \ \text{on}\ \ \Pi(G(\d/2)).\]
In correspondence of each $\d>0$, we have
\[\begin{array}{l}
\ds{{\mathbb E}_z(\varphi^\wedge)^\vee(Z^\e(t))-\bar{{\mathbb E}}_{\Pi(z)}\varphi^\wedge(\bar{Z}(t))}\\
\vs
\ds{={\mathbb E}_z\le[(\varphi^\wedge)^\vee(Z^\e(t))-\psi_\d^\vee(Z^\e(t))\r]+\le[{\mathbb E}_z\,\psi_\d^\vee(Z^\e(t))-\bar{{\mathbb E}}_{\Pi(z)}\psi_\d(\bar{Z}(t))\r]}\\
\vs
\ds{+\bar{{\mathbb E}}_{\Pi(z)}\le[\psi_\d(\bar{Z}(t))-\varphi^\wedge(\bar{Z}(t))\r]=:I^{\e,\d}_1(t)+I_2^{\e,\d}(t)+I^\d(t).}
\end{array}\]
If we can show that for any $\d>0$ there exists some $\e_\d>0$ such that
\begin{equation}
\label{cf101}
\lim_{\d\to 0}\sup_{\e \in\,(0,\e_\d)}\,\sup_{t \in\,[\tau,T]}\,|I^{\e,\d}_1(t)|=0,
\end{equation}
and
\begin{equation}
\label{cf102}
\lim_{\d\to 0}\,\sup_{t \in\,[\tau,T]}\,|I^\d(t)|=0,
\end{equation}
then for any $\eta>0$ we can find $\d_\eta>0$ and $\e_\eta>0$ such that
\begin{equation}
\label{cf103}
\sup_{t \in\,[\tau,T]}\,\le|{\mathbb E}_z(\varphi^\wedge)^\vee(Z^\e(t))-\bar{{\mathbb E}}_{\Pi(z)}\varphi^\wedge(\bar{Z}(t))\r|\leq \eta+\,\sup_{t \in\,[\tau,T]}\le|I^{\e,\d_\eta}_2(t)\r|,\ \ \ \ \e \leq \e_\eta.\end{equation}
Since $\psi_\d \in\,C(\bar{\Gamma})$, due to \eqref{cf1000} we have
\[\lim_{\e\to 0}\,\sup_{t \in\,[\tau,T]}\,|I^{\e,\d_\eta}_2(t)|=0,\]
 and then, due to the arbitrariness of $\eta$, from \eqref{cf103} we can conclude that \eqref{cf100} holds.
 
 In order to prove \eqref{cf101}, we write
 \[\begin{array}{l}
 \ds{I^{\e,\d}_1(t)={\mathbb E}_z\le(\Delta_{\e,\d}(t)\,;\,t \in\,\bigcup_{n \in\,\nat}\le[\tau_n^{\e,\d,\d/2},\si_n^{\e,\d,\d/2}\r)\r)+{\mathbb E}_z\le(\Delta_{\e,\d}(t)\,;\,t \in\,\bigcup_{n \in\,\nat}\le[\si_n^{\e,\d,\d/2},\tau_{n+1}^{\e,\d,\d/2}\r)\r),}
 \end{array}\]
 where
 \[\Delta_{\e,\d}(t):=(\varphi^\wedge)^\vee(Z^\e(t))-\psi_\d^\vee(Z^\e(t)).\]
 Recalling that $\varphi^\wedge\equiv \psi_\d$ on $\Pi(G(\d/2))$, this yields 
\[\begin{array}{l}
\ds{I^{\e,\d}_1(t)={\mathbb E}_z\le(\Delta_{\e,\d}(t)\,;\,t \in\,\bigcup_{n \in\,\nat}\le[\tau_n^{\e,\d,\d/2},\si_n^{\e,\d,\d/2}\r)\r)}\\
\vs
\ds{=\sum_{n \in\,\nat}{\mathbb E}_z\le(\Delta_{\e,\d}(t)\,I_{\{\tau_n^{\e,\d,\d/2}\leq t\}}I_{\{ \si_n^{\e,\d,\d/2}>t\}}\r):=\sum_{n \in\,\nat} J^{\e,\d}_{1,n}(t).}
\end{array}\]
Due to the strong Markov property, 
\begin{equation}
\label{cf105}
\begin{array}{l}
\ds{|J^{\e,\d}_{1,n}(t)|\leq {\mathbb E}_z\le(I_{\{\tau_n^{\e,\d,\d/2}\leq t\}}\le|{\mathbb E}_{Z^\e(\tau_n^{\e,\d,\d/2})}\le(\Delta_{\e,\d}(t)\,;\, \si_0^{\e,\d,\d/2}>t\r)\r|\r)}\\
\vs
\ds{\leq {\mathbb P}_z\le(\tau_n^{\e,\d,\d/2}\leq t\r)\La_{\e,\d}(t)\leq e^t\, {\mathbb E}_z\le(e^{-\tau_n^{\e,\d,\d/2}}\r)\La_{\e,\d}(t),}
\end{array}\end{equation}
where
\[\La_{\e,\d}(t):=\sup_{z \in\,C(\d/2)}\le|{\mathbb E}_z\le(\Delta_{\e,\d}(t)\,;\, \si_0^{\e,\d,\d/2}>t\r)\r|.\]

Now, in \cite[Lemma 3.10]{fw12} it is proven that if $\tau^\e=\tau^\e(a_k+\d^\prime,b_k-\d^\prime)$ is the first time the process $X^\e(t)$ leaves the interval $(a_k+\d^\prime,b_k-\d^\prime)$, then for any $\la>0$
\[\lim_{\d\to 0}\frac 1\d \lim_{\d^\prime \to 0}\lim_{\e\to 0} \E_{(a_k+\d,y)}\le(1-\exp\le(-\la \tau^\e\r)\r)>0,\]
and
\[\lim_{\d\to 0}\frac 1\d \lim_{\d^\prime \to 0}\lim_{\e\to 0} \E_{(b_k-\d,y)}\le(1-\exp\le(-\la \tau^\e\r)\r)>0,\]
uniformly with respect to the points $(a_k+\d,y)$ and $(b_k-\d,y)$ in $\Pi^{-1}(I_k)$. This implies that there exist $0<\d_1^\prime<\d_1$, $\bar{\rho}<\d_1^{-1}$ and $\e_1>0$ such that for every $\e\leq \e_1$, $\d\leq \d_1$ and $0<\d^\prime<\d^\prime_1\wedge \d$  
\[\sup_{z \in\,C(\d) }\E_z\exp\le(-\tau^{\e,\d,\d^\prime}_1\r)\leq (1-\bar{\rho}\d).\]
Due to the strong Markov property, this yields
\begin{equation}
\label{strong}
\sup_{z \in\,G }\,\E_z\exp\le(-\tau^{\e,\d,\d^\prime}_n\r)\leq (1-\bar{\rho}\d)^n,
\end{equation}
for every $\e\leq \e_1$, $\d\leq\d_1$ and $0<\d^\prime<\d^\prime_1\wedge \d$. This implies in particular that
\begin{equation}
\label{cf110}
|I^{\e,\d}_{1}(t)|\leq e^t \La_{\e,\d}(t)\sum_{n \in\,\nat} (1-\bar{\rho}\,\d)^{n}\leq e^T \La_{\e,\d}(t)\,\frac 1{\d\bar{\rho}}, \ \ \ \ t \in\,[0,T].
\end{equation}

Moreover, in \cite[Lemma 6.2]{fw12}  it is proven that, for some $\d_2>0$ , for every $0<\d\leq\d_2$ there exists $\e_\d>0$ such that for all $\e \in\,(0,\e_\d]$ and all $z \in\,G(O_i,x_i-\d,x_i+\d)$, with $i=1,\ldots,N$, we have
\begin{equation}
\label{cf300}
\E_z \tau^\e(G(O_i,x_i-\d,x_i+\d))\leq 5\,\d^2,\end{equation}
where $\tau^\e(G(O_i,x_i-\d,x_i+\d))$ is the first exit time of $Z^\e_z$ from $G(O_i,x_i-\d,x_i+\d)$.

Therefore, if we set $\bar{\d}:=\d_1\wedge \d_2$, then for any $\d\leq \bar{\d}$, there exists $\e_\d>0$  such that
\[\La_{\e,\d}(t)\leq \frac 2 t\|\varphi^\wedge\|_\infty \sup_{z \in\,G(O_i,x_i-\d,x_i+\d)} {\mathbb E}_z\si_0^{\e,\d,\d/2}\leq \frac {c}{\tau}\|\varphi\|_\infty \d^2,\ \ \ \ t \in\,[\tau,T].\]
This, together with \eqref{cf110}, implies \eqref{cf101}.

Now, for any $\d>$ and $t\geq 0$, we have
\[|I^\d(t)|\leq 2\|\varphi\|_\infty\,\bar{\mathbb{P}}_{\Pi(z)}\le(\bar{Z}(t) \in\,\Pi(G(\d/2))^c\r)\leq 2\|\varphi\|_\infty\,\bar{\mathbb{E}}_{\Pi(z)}\le(f_\d(\bar{Z}(t))\r),\]
for some $f_\d \in\,C(\bar{\Gamma})$ such that 
$I_{\Pi(G(\d/2))^c}\leq f_\d\leq 1$ and $f_\d\equiv 0$ on $\Pi(G(\d))$.
This yields
\[|I^\d(t)|\leq 2\|\varphi\|_\infty\le|\bar{\mathbb{E}}_{\Pi(z)}f_\d(\bar{Z}(t))-\mathbb{E}_{z}f_\d^\vee(Z_\e(t))\r|+2\,\|\varphi\|_\infty\,\le|\mathbb{E}_{z}f_\d^\vee(Z_\e(t))\r|.\]
According to \eqref{cf101}, for any $\eta>0$ there exists $\d_\eta, \e_\eta >0$ such that
\[\sup_{t \in\,[\tau,T]}\le|\mathbb{E}_{z}f_\d^\vee(Z_\e(t))\r|<\eta,\ \ \ \ \e\leq \e_\eta.\]
Then, for any $\e\leq \e_\eta$
\[\sup_{t \in\,[\tau,T]}|I^\d(t)|\leq 2\|\varphi\|_\infty\,\eta+2\|\varphi\|_\infty\,\sup_{t \in\,[\tau,T]}\le|\bar{\mathbb{E}}_{\Pi(z)}f_\d(\bar{Z}(t))-\mathbb{E}_{z}f_\d^\vee(Z_\e(t))\r|,\]
so that, due to \eqref{cf1000} and the arbitrariness of $\eta$, we get \eqref{cf102}.

\end{proof}

Now we can prove the main result of this section.

\begin{Theorem}
\label{theorem4.1}
If the domain $G$ satisfies assumptions I-IV, then for any $\varphi \in\,C(\bar{G})$ and $z \in\,G$  and for any $0\leq \tau\leq T$ we have
\begin{equation}
\label{funda}
\lim_{\e\to 0}\,\sup_{t \in\,[\tau,T]}\,\le|\E_z\,\varphi(Z^\e(t))-\bar{\E}_{\Pi(z)}\,\varphi^\wedge(\bar{Z}(t))\r|=0.
\end{equation}

\end{Theorem}

\begin{proof} In Lemma \ref{lemf} we have proven that for any $t>0$ and $z \in\,G$
\[\lim_{\e\to 0}\,\sup_{t \in\,[\tau,T]}\,\le|\E_z\,(\varphi^\wedge)^\vee(Z^\e(t))-\bar{\E}_{\Pi(z)}\,\varphi^\wedge(\bar{Z}(t))\r|=0.\]
Thus, in order to prove \eqref{funda}, it is sufficient to show that
\begin{equation}
\label{cf6}
\lim_{\e\to 0}\,\sup_{t \in\,[\tau,T]}\,\le|\E_z\,\le(\varphi(Z^\e(t))-(\varphi^\wedge)^\vee(Z^\e(t))\r)\r|=0.
\end{equation}

In what follows we can assume that $\varphi \in\,\text{Lip}(\bar{G})$. Actually, for any $\varphi \in\,C(\bar{G})$ there exists $\{\varphi_n\}_{n \in\,\nat}\subset \text{Lip}(\bar{G})$ such that
\[\lim_{n\to \infty}\|\varphi_n-\varphi\|_\infty=0.\]
As this implies
\[\lim_{n\to \infty}\|(\varphi_n^\wedge)^\vee-(\varphi^\vee)^\wedge\|_\infty=0,\]
we obtain that 
\[\lim_{n\to\infty}\,\sup_{t \in\,[\tau,T]}\le(\le|\E_z\,\le(\varphi(Z^\e(t))-\varphi_n(Z^\e(t))\r)\r|+\le|\E_z\,\le((\varphi^\wedge)^\vee(Z^\e(t))-(\varphi_n^\wedge)^\vee(Z^\e(t))\r)\r|\r)=0,\]
uniformly with respect to $\e>0$ and $t\geq 0$.
Hence, for any $\eta>0$ there exists $n_\eta \in\,\nat$ such that
\[\sup_{t \in\,[\tau,T]}\le|\E_z\,\le(\varphi(Z^\e(t))-(\varphi^\wedge)^\vee(Z^\e(t))\r)\r|\,\leq \eta+\sup_{t \in\,[\tau,T]}\le|\E_z\,\le(\varphi_{n_\eta}(Z^\e(t))-(\varphi_{n_\eta}^\wedge)^\vee(Z^\e(t))\r)\r|.\]

For any fixed $t>0$, we can assume that the partition introduced in \eqref{discrete} and in the proofs of Lemma \ref{lemmaA.2} and  of Theorem \ref{lim-z-eps}, where we have defined the approximating process $\hat{Z}^{\e}=\hat{Z}^{\e,\gamma_\e}$, is such that 
\[t=\le(k^\e_t+\frac 12\r)\gamma_\e,\]
where $\gamma_\e$ is the positive constant defined in Theorem \ref{lim-z-eps} and $k^\e_t \in\,\nat$. Notice that we can take $\e>0$ small enough so that 
$\gamma_\e<\tau$ and hence, as $t\geq \tau$, \begin{equation}
\label{cf9}
\frac t2<k^\e_t\gamma_\e<t.\end{equation}
Moreover, with the notations introduced in Section \ref{app1}, we have
\[\hat{Z}^\e(t)=(X^\e(k^\e_t\gamma_\e),Y^{\e,\gamma_\e}(t)),\]
and, because of the way $Y^{\e,\gamma_\e}(t)$ has been defined, we have 
\[Z^\e(k^\e_t\gamma_\e)=(X^\e(k^\e_t\gamma_\e),Y^\e(k^\e_t\gamma_\e)) \in\,G(\d)\Longrightarrow \hat{Z}^\e(t) \in\,G(\d).\]

In the proof of \eqref{cf6} we will proceed in two steps.

\medskip

{\em Step 1.} We show that for some  $0<\bar{\d}^\prime<\bar{\d}$ and any $\d\leq \bar{\d}$ and $\d^\prime < \bar{\d}^\prime\wedge \d$ there exists $\e_\d>0$ such  that for any  $z \in\,G$ and $\e\leq \e_\d$ it holds
\begin{equation}
\label{step1}
\le|\E_z\,\le(\varphi(Z^\e(t))-(\varphi^\wedge)^\vee(Z^\e(t))\r)\r|\leq c\,e^{T}\le(\frac{\|\varphi\|_\infty}{\tau}\,\d+\frac 1\d\, L^\e_{\d,\d^\prime}(t)\r),\ \ \ \ t \in\,[\tau,T],
\end{equation}
where
\begin{equation}
\label{leps}
L^\e_{\d,\d^\prime}(t):=\sup_{z \in\,C(\d)}\le|\E_z\,\le(\varphi(Z^\e(t))-(\varphi^\wedge)^\vee(Z^\e(t))\,;\,\tau^{\e,\d,\d^\prime}_1>k^\e_t\gamma_\e \r)\r|.\end{equation}

\smallskip

If we define
\[\Delta_\e(t):=\varphi(Z^\e(t))-(\varphi^\wedge)^\vee(Z^\e(t)),\]
 for any $0<\d^\prime<\d$ and $\e>0$ we have
\[\begin{array}{l}
\ds{\E_z\,\le(\varphi(Z^\e(t))-(\varphi^\wedge)^\vee(Z^\e(t))\r)}\\
\vs
\ds{=\E_z\le(\Delta_\e(t)\,;\,k^\e_t\gamma_\e \in\, \bigcup_{n \in\,\nat}[\tau^{\e,\d,\d^\prime}_n,\si^{\e,\d,\d^\prime}_n)\r)+\E_z\le(\Delta_\e(t)\,;\,k^\e_t\gamma_\e \in\, \bigcup_{n \in\,\nat}[\si^{\e,\d,\d^\prime}_n,\tau^{\e,\d,\d^\prime}_{n+1})\r)}\\
\vs
\ds{= \sum_{n \in\,\nat}\E_z\le(I_{\{\tau^{\e,\d,\d^\prime}_n\leq k^\e_t\gamma_\e\}}\,I_{\{\si^{\e,\d,\d^\prime}_n> k^\e_t\gamma_\e\}}\Delta_\e(t)\r)+\sum_{n \in\,\nat}\E_z\le(I_{\{\si^{\e,\d,\d^\prime}_n\leq k^\e_t\gamma_\e\}}\,I_{\{\tau^{\e,\d,\d^\prime}_{n+1}> k^\e_t\gamma_\e\}}\Delta_\e(t)
\r)}\\
\vs
\ds{=:\sum_{n \in\,\nat} J^{\e, \d,  \d^\prime}_{1,n}(t)+\sum_{n \in\,\nat}J^{\e, \d,  \d^\prime}_{2,n}(t).}
\end{array}\]

As a consequence of the strong Markov property, for each $n \in\,\nat$ we have
\begin{equation}
\label{j1}
J^{\e, \d,  \d^\prime}_{1,n}(t)=\E_z\le(I_{\{\tau^{\e,\d,\d^\prime}_n\leq k^\e_t\gamma_\e\}}\,\E_{Z^\e(\tau^{\e,\d,\d^\prime}_n)}\le(\Delta_\e(t)\,;\,\si^{\e,\d,\d^\prime}_0>k^\e_t\gamma_\e\r)\r).
\end{equation}
Hence, thanks to \eqref{cf300},  for any $\d>0$ sufficiently small there exists $\e_\d>0$ such that for any $\e \in\,(0,\e_\d)$ 
\[\begin{array}{l}
\ds{\le|\E_{Z^\e(\tau^{\e,\d,\d^\prime}_n)}\le(\Delta_\e(t)\,;\,\si^{\e,\d,\d^\prime}_0>k^\e_t\gamma_\e\r)\r|\leq \frac{2\,\|\varphi\|_\infty}{k^\e_t\gamma_\e}\sup_{z \in\,G(O_i,x_i-\d,x_i+\d)}\E_z\,\si^{\e,\d,\d^\prime}_0\leq \frac{10\,\|\varphi\|_\infty}{k^\e_t\gamma_\e}\,\d^2.}
\end{array}\]
Therefore, due to \eqref{j1}, we have
\begin{equation}
\label{j1-1}
|J^{\e, \d,  \d^\prime}_{1,n}(t)|\leq \frac{10\,\|\varphi\|_\infty}{k^\e_t\gamma_\e}\,\d^2\,{\mathbb{P}}_z(\tau^{\e,\d,\d^\prime}_n\leq k^\e_t\gamma_\e)\leq \frac{10\,\|\varphi\|_\infty\,e^{k^\e_t\gamma_\e}}{k^\e_t\gamma_\e}\,\d^2\,\E_z\exp\le(-\tau^{\e,\d,\d^\prime}_n\r).
\end{equation}

As in the proof of Lemma \ref{lemf}, according to \eqref{j1-1}   and \eqref{cf9}, estimate \eqref{strong} implies that for any $0<\d^\prime<\d$ sufficiently small, there exists $\e_\d>0$  such that for any $\e\leq \e_\d$ 
\begin{equation}
\label{ji-2}
\sum_{n \in\,\nat}\,|J^{\e, \d,  \d^\prime}_{1,n}(t)|\leq \frac{c\,\|\varphi\|_\infty\,e^{k^\e_t \gamma_\e}}{k^\e_t \gamma_\e}\,\d^2\sum_{n \in\,\nat}(1-\bar{\rho}\d)^n\leq \frac{c\,\|\varphi\|_\infty\,e^t}{t}\,\frac{\d}{\bar{\rho}}\leq c\,\|\varphi\|_\infty \frac{e^T}{\tau}\,\d.
\end{equation}

Now, let us study $J^{\e,2}_n(t)$. From the strong Markov property, we have
\[\begin{array}{l}
\ds{|J^{\e, \d,  \d^\prime}_{2,n}(t)|\leq \E_z\le(I_{\{\si^{\e,\d,\d^\prime}_n\leq  k^\e_t \gamma_\e\}}\,\le|\E_{Z^\e(\si^{\e,\d,\d^\prime}_n)}\le(\Delta_\e(t)\,;\,\tau^{\e,\d,\d^\prime}_1>k^\e_t \gamma_\e \r)\r|\r)}\\
\vs
\ds{\leq \Pro_z\le(\si^{\e,\d,\d^\prime}_n\leq k^\e_t \gamma_\e\r)   L^{\e}_{\d,\d^\prime}\leq e^t\,\E_z\,\exp\le(-\tau^{\e,\d,\d^\prime}_n\r)\,L^{\e}_{\d,\d^\prime}(t),}
\end{array}\]
where $L^\e_{\d,\d^\prime}(t)$ is the function defined in \eqref{leps}.
Hence, thanks to \eqref{strong}, we get
\[\sum_{n \in\,\nat}\,|J^{\e, \d,  \d^\prime}_{2,n}(t)|\leq \sum_{n \in\,\nat}(1-\bar{\rho}\d)^n\,e^T\,L^\e_{\d,\d^\prime}(t)\leq \frac{e^T}{\d \bar{\rho}}\,L^\e_{\d,\d^\prime}(t),\]
for every $\e\leq \e_1$, $\d\leq\d_1$ and $0<\d^\prime<\d^\prime_1\wedge \d$. This, together with \eqref{ji-2}, implies \eqref{step1}.

\medskip

{\em Step 2.}  For any $0<\d^\prime<\d$, it holds
\begin{equation}
\label{step2}
\lim_{\e\to 0} \,\sup_{t \in\,[\tau,T]}\,L^\e_{\d,\d^\prime}(t)=0.
\end{equation}

\smallskip

For any $\e>0$ we have
\[\begin{array}{l}
\ds{\varphi(Z^\e(t))-(\varphi^\wedge)^\vee(Z^\e(t))=\le[\varphi(Z^\e(t))-\varphi(\hat{Z}^\e(t))\r]}\\
\vs
\ds{+\le[\varphi(\hat{Z}^\e(t))-(\varphi^\wedge)^\vee(\hat{Z}^\e(t))\r]+\le[(\varphi^\wedge)^\vee(\hat{Z}^\e(t))-(\varphi^\wedge)^\vee(Z^\e(t))\r]=:\sum_{i=1}^3I^\e_i(t).}
\end{array}\]

If we denote by $h_z$ the integer such that $z \in\,\Pi^{-1}(I_{h_z})$, we have
\[\begin{array}{l}
\ds{\E_z\le(|I^\e_1(t)|\ ;\  \tau^{\e,\d,\d^\prime}_1>k^\e_t \gamma_\e \r)=\E_z\le(|I^\e_1(t)|\ ;\ \tau^{\e,\d,\d^\prime}_1>k^\e_t \gamma_\e\, ,\, \tau^{\e,\d,\d^\prime/2}_1>(k^\e_t+1/2) \gamma_\e \r)}\\
\vs
\ds{+\E_z\le(|I^\e_1(t)|\ ;\  \tau^{\e,\d,\d^\prime}_1>k^\e_t \gamma_\e\, ,\, \tau^{\e,\d,\d^\prime/2}_1\leq(k^\e_t+1/2) \gamma_\e \r)}\\
\vs
\ds{\leq \E_z\le(|I^\e_1(t)|\ ;\ Z^\e(t),\ \hat{Z}^\e(t) \in\,G_{h_z}(a_{k_z}+\d^\prime/2,b_{h_z}-\d^\prime/2)\r)}\\
\vs
\ds{+2\,\|\varphi\|_\infty\Pro_z\le(\tau^{\e,\d,\d^\prime}_1>k^\e_t \gamma_\e\, ,\, \tau^{\e,\d,\d^\prime/2}_1\leq (k^\e_t+1/2) \gamma_\e \r).}
\end{array}\]
Now, according to Theorem \ref{lim-z-eps}, since we are assuming $\varphi \in\,\text{Lip}(\bar{G})$, we have that
\begin{equation}
\label{cf10}
\lim_{\e\to 0}\sup_{z \in\,C(\d)}\,\sup_{t \in\,[\tau,T]}\,\E_z\le(|I^\e_1(t)|\ ;\ Z^\e(t),\ \hat{Z}^\e(t) \in\,G_{h_z}(a_{h_z}+\d^\prime/2,b_{h_z}-\d^\prime/2)\r)=0.\end{equation}
Moreover,
\[\begin{array}{l}
\ds{\Pro_z\le(\tau^{\e,\d,\d^\prime}_1> k^\e_t \gamma_\e\, ,\, \tau^{\e,\d,\d^\prime/2}_1\leq (k^\e_t+1/2) \gamma_\e \r)}\\
\vs
\ds{\leq \Pro_z\le(|X^\e(\tau_1^{\e,\d,\d^\prime/2})-X^\e (k^\e_t \gamma_\e)|\geq \d^\prime/2\, ,\, \tau^{\e,\d,\d^\prime}_1> k^\e_t \gamma_\e\, ,\, \tau^{\e,\d,\d^\prime/2}_1\leq(k^\e_t+1/2) \gamma_\e \r)}\\
\vs
\ds{\leq \Pro_z\le(|X^\e(t\wedge \tau_1^{\e,\d,\d^\prime/2})-X^\e (k^\e_t \gamma_\e\wedge \tau_1^{\e,\d,\d^\prime/2} )|\geq \d^\prime/2\r)}\\
\vs
\ds{\leq \le(\frac 2{\d^\prime}\r)^2\E_z\,|X^\e(t\wedge \tau_1^{\e,\d,\d^\prime/2})-X^\e (k^\e_t \gamma_\e\wedge \tau_1^{\e,\d,\d^\prime/2} )|^2.}
\end{array}\]
Since 
\[\begin{array}{l}
\ds{X^\e(t\wedge \tau_1^{\e,\d,\d^\prime/2})-X^\e (k^\e_t \gamma_\e\wedge \tau_1^{\e,\d,\d^\prime/2} )}\\
\vs
\ds{=B_1(t\wedge \tau_1^{\e,\d,\d^\prime/2})-B_1(k^\e_t \gamma_\e\wedge \tau_1^{\e,\d,\d^\prime/2} )+\int_{k^\e_t \gamma_\e\wedge \tau_1^{\e,\d,\d^\prime/2}}^{t\wedge \tau_1^{\e,\d,\d^\prime/2}}\nu_1(X^\e(s),Y^\e(s))\,d\phi^\e(s),}
\end{array}\]
from \eqref{phi} we get
\[\E|X^\e(t\wedge \tau_1^{\e,\d,\d^\prime/2})-X^\e (k^\e_t \gamma_\e\wedge \tau_1^{\e,\d,\d^\prime/2} )|^2\leq c\,\gamma_\e,\]
so that
\[\Pro_z\le(\tau^{\e,\d,\d^\prime}_1> k^\e_t \gamma_\e\, ,\, \tau^{\e,\d,\d^\prime/2}_1\leq(k^\e_t+1/2) \gamma_\e \r)\leq c\,\frac {\gamma_\e}{(\d^\prime)^2}.\]
This, together with \eqref{cf10},
implies  that
\begin{equation}
\label{I1}
\lim_{\e\to 0}\sup_{z \in\,C(\d)}\,\sup_{t \in\,[\tau,T]}\,\E_z\le(|I^\e_1(t)|\,;\,\tau^{\e,\d,\d^\prime}_1>k^\e_t \gamma_\e\r)=0.
\end{equation}
As $(\varphi^\wedge)^\vee$ is continuous in  $G(\d)$, for any $\d>0$, we can repeat the same arguments used for $I^\e_1(t)$ to prove that 
\begin{equation}
\label{cf11}
\lim_{\e\to 0}\sup_{z \in\,C(\d)}\,\sup_{t \in\,[\tau,T]}\,\E_z\le(|I^\e_3(t)|\,;\,\tau^{\e,\d,\d^\prime}_1>k^\e_t \gamma_\e\r)=0.
\end{equation}

Now, it remains  to study 
\[\E_z\le(I^\e_2(t)\,;\,\tau^{\e,\d,\d^\prime}_1>k^\e_t \gamma_\e\r).\]
As a consequence of the Markov property, we have
\[\E_z\le(I^\e_2(t)\,;\,\tau^{\e,\d,\d^\prime}_1>k^\e_t \gamma_\e\r)=\E_z\le(\psi(\gamma_\e/2,Z^\e(k^\e_t \gamma_\e))\,;\,\tau^{\e,\d,\d^\prime}_1>k^\e_t \gamma_\e\r),\]
where
\[\psi(s,x,y)=\E_{(x,y)}\varphi(x,Y^\e_{1,(x,y)}(s))-\varphi^\wedge(x,h_{(x,y)}),\]
 $h_{(x,y)}$ is the integer defined by $(x,y) \in\,\Pi^{-1}(I_{h_{(x,y)}})$, and $Y^\e_{1,(x,y)}(s)$ is the process defined, as in \eqref{cf12}, by the equation
\[dY^\e_{1,(x,y)}(s)=\frac 1{\e}\,dB(s)+\frac 1{\e^2}\,\nu_2(x,Y^\e_{1,(x,y)}(s))\,d\phi^\e_1(s),\ \ \ \ s \in\,[0,\gamma_\e),\ \ \ \ Y^\e_{1,(x,y)}(0)=y.\]
Now, by proceeding as in the proof of Lemma \ref{lemmaA.2}, we have that
\[\psi(s,x,y)=\tilde{\E}_{(x,y)}\varphi(x,Y_{2,(x,y)}(s/\e^2))-\varphi^\wedge(x,h_{(x,y)}),\]
where $Y_{2,(x,y)}(s)$ is defined, as in \eqref{cf13}, by the equation
\[dY_{2,(x,y)}(s)=d\tilde{B}(s)+\nu_2(x,Y_{2,(x,y)}(s))\,d\phi_2(t),\ \ \ \ s \in\,[0,\gamma_\e/\e^2),\ \ \ \ \ Y_{2,(x,y)}(0)=y.\]

Notice that, due to our assumptions on the domain $G$,  for any $\d>0$ and $k=1,\ldots,N$,
\[G_{k}(a_k+\d,b_k-\d)=\{(x,y) \in\,\reals^2\,:\,h_{1,k}(x)\leq y\leq h_{2,k}(x),\ x \in\,(a_k+\d,b_k-\d)\},\]
for some smooth functions $h_{1,k}(x)$ and $h_{2,k}(x)$, and, if
$l_k(x)$ is the length of the cross-section  $C_k(x)=\{(x,y) \in\,G_{k}(a_k+\d,b_k-\d)\}$, we have
\[\inf_{x \in\,(a_k+\d,b_k-\d)}l_k(x)=\inf_{x \in\,(a_k+\d,b_k-\d)}(h_{2,k}(x)-h_{1,k}(x))=:l_{k,\d}>0.\]
For any $(x,y) \in\,G_{k}(a_k+\d,b_k-\d)$, the process $Y_{2,(x,y)}(s)$ lives in the interval $[h_{1,k}(x),h_{2,k}(x)]$. Because of the way the process $Y_{2,(x,y)}(s)$ has been defined, for any $f \in\,C([h_{1,k}(x),h_{2,k}(x)])$ we have
 \[\tilde{\E}_{(x,y)}f(Y_{2,(x,y)}(s))=\sum_{j=0}^\infty e^{-s\a_{k,j}(x)}\le<f,e^x_{k,j}\r>e^x_{k,j},\ \ \ y \in\,[h_{1,k}(x),h_{2,k}(x)],\]
 where
 \[e^x_{k,0}(y)=\frac 1{\sqrt{l_k(x)}},\ \ \ e^x_{k,j}(y)=\sqrt{\frac{2}{l_k(x)}}\cos\le(\frac{j\pi}{l_k(x)}(y-h_{1,k}(x))\r),\ \ \ j=1,2,\ldots,\]
 and
 \[\a_{k,0}(x)=0,\ \ \ \a_{k,j}(x)=-\le(\frac{j\pi}{l_k(x)}\r)^2,\ \ \ j=1,2,\ldots.\]
Recalling how $\varphi^\wedge$ is defined, this implies
\[\psi(s,x,y)= \sum_{j=1}^\infty e^{-\frac{s}{\e^2}\a_{k,j}(x)}\le<\varphi(x,\cdot),e^x_{k,j}\r>e^x_{k,j}(y),\]
so that for any $rho>1/2$ we have
\[\begin{array}{l}
\ds{|\psi(s,x,y)|^2\leq \frac{2}{l_k(x)}\sum_{j=1}^\infty e^{-\frac{2s}{\e^2}\a_{k,j}(x)}\,\sum_{j=1}^\infty \le|\le<\varphi(x,\cdot),e^x_{k,j}\r>\r|^2}\\
\vs
\ds{\leq c_\rho\,\|\varphi\|^2_{\infty} \le(\frac s{\e^2}\r)^{-\rho}\sum_{j=1}^\infty \a_{k,j}(x)^{-\rho}=c_\rho\le(\frac{l_{h_{(x,y)}}(x)}{\pi}\r)^\rho\,\|\varphi\|^2_{\infty} \le(\frac s{\e^2}\r)^{-\rho}\sum_{j=1}^\infty  k^{-2\rho}\leq c_\rho \,\|\varphi\|_\infty^2\,\le(\frac s{\e^2}\r)^{-\rho}.}
\end{array}\]
Therefore, we can can conclude that
\[\begin{array}{l}
\ds{\le|\E_z\le(I^\e_2(t)\,;\,\tau^{\e,\d,\d^\prime}_1>k^\e_t \gamma_\e\r)\r|\leq \E_z\le(\le|\psi(\gamma_\e/2,Z^\e(k^\e_t \gamma_\e))\r|\,;\,\tau^{\e,\d,\d^\prime}_1>k^\e_t \gamma_\e\r)}\\
\vs
\ds{\leq c_\rho \,\|\varphi\|_\infty\,\le(\frac {\e^2}{\gamma_\e}\r)^{\frac{\rho}2}.}
\end{array}\]
As we are assuming that $\e^2/\gamma_\e\to 0$, as $\e\to 0$ (see Theorem \ref{lim-z-eps}), we can conclude that 
\[\lim_{\e\to 0}\,\sup_{t \in\,[\tau,T]}\, \E_z\le(I^\e_2(t)\,;\,\tau^{\e,\d,\d^\prime}_1>k^\e_t \gamma_\e\r)=0,\] and, together with 
\eqref{I1} and \eqref{cf11}, this yields \eqref{step2}.

\medskip

{\em Conclusion.} Due to \eqref{step1}, for any $\eta>0$ we can fix $0<\d^\prime_\eta<\d_\eta$ and $\e_\eta>0$ such that for any $\e\leq \e_{\eta}$
\[\le|\E_z\,\le[\varphi(Z^\e(t))-(\varphi^\wedge)^\vee(Z^\e(t))\r]\r|\leq \eta+\frac 1{\d_\eta}\, L^\e_{\d_\eta,\d^\prime_\eta}(t),\ \ \ \ t \in\,[\tau,T].\]
Thus, according to \eqref{step2}, due to the arbitrariness of $\eta$ we can conclude that  \eqref{cf6} holds true,
and \eqref{funda} follows.

\end{proof}

In Sections \ref{sec1} and \ref{sec2}, we have introduced the semigroups $\bar{S}(t)$ and $S_\e(t)$, associated respectively with the operators $\bar{L}$ and $L_\e$. With these notations, as a consequence of \eqref{limite}, we have that for any $f \in\,C(\bar{\Gamma})$, $z \in\,G$ and $t>0$
\begin{equation}
\label{cf205}
\lim_{\e\to 0}S_\e(t) f^\vee (z)=(\bar{S}(t)f)^\vee(z).
\end{equation}
Now, in view of Theorem \ref{theorem4.1} we get also the following limit result.
\begin{Corollary}
Under Hypotheses I-IV for the domain $G$, for any $0\leq \tau<T$, $\varphi \in\,C(\bar{G})$ and $z \in\,G$, we have
\begin{equation}
\label{fundacor}
\lim_{\e\to 0}\,\sup_{t \in\,[\tau,T]}\,|S_\e(t)\varphi(z)-\bar{S}(t)^\vee\varphi(z)|=0.
\end{equation}
Moreover, for any $\varphi \in\,H$
\begin{equation}
\label{cf18}
\lim_{\e\to 0}\,\sup_{t \in\,[\tau,T]}\,|S_\e(t)\varphi-\bar{S}(t)^\vee\varphi|_H=\lim_{\e\to 0}\,\sup_{t \in\,[\tau,T]}\,|(S_\e(t)\varphi)^\wedge-\bar{S}(t)\varphi^\wedge|_{\bar{H}}=0.
\end{equation}
\end{Corollary}
\begin{proof}
Since $\bar{S}(t)^\vee\varphi=(\bar{S}(t)\varphi^\wedge)^\vee$, limit \eqref{fundacor} is an immediate consequence of \eqref{funda}. Moreover, as 
\[\sup_{z \in\,G}\,\sup_{t \in\,[\tau,T]}\,|S_\e(t)\varphi(x)|=\sup_{z \in\,G}\,\sup_{t \in\,[\tau,T]}\,|\E_z\,\varphi(Z^\e(t))|\leq \|\varphi\|_\infty,\]
and
\[\begin{array}{l}
\ds{\sup_{z \in\,G}\,\sup_{t \in\,[\tau,T]}\,|(\bar{S}(t)\varphi^\wedge)^\vee(z)|=\sup_{(x,k) \in\,\Gamma}\,\sup_{t \in\,[\tau,T]}\,|\bar{S}(t)\varphi^\wedge(x,k)|}\\
\vs
\ds{=\sup_{(x,k) \in\,\Gamma}\,\sup_{t \in\,[\tau,T]}\,|\bar{\mathbb{E}}_{(x,k)}\varphi^\wedge(\bar{Z}(t)|\leq \|\varphi^\wedge\|_\infty\leq \|\varphi\|_\infty,}
\end{array}\]
by the dominated convergence theorem, from \eqref{fundacor} we get \eqref{cf18} for any $\varphi \in\,C(\bar{G})$. Now, if $\varphi \in\,H$, for any $\eta>0$ we can find $\bar{\varphi} \in\,C(\bar{G})$ such that $|\varphi-\bar{\varphi}|_H\leq \eta/4$.
This implies 
\[\begin{array}{l}
\ds{|S_\e(t)\varphi-\bar{S}(t)^\vee\varphi|_H\leq |S_\e(t)(\varphi-\bar{\varphi})-\bar{S}(t)^\vee(\varphi-\bar{\varphi})|_H+|S_\e(t)\bar{\varphi}-\bar{S}(t)^\vee\bar{\varphi}|_H}\\
\vs
\ds{\leq \frac \eta 2+|S_\e(t)\bar{\varphi}-\bar{S}(t)^\vee\bar{\varphi}|_H,}
\end{array}\]
so that we can find $\e_\eta>0$ such that 
\[|S_\e(t)\varphi-\bar{S}(t)^\vee\varphi|_H\leq \frac{\eta}2,\ \ \ \ \e\leq \e_\eta.\]
Due to the arbitrariness of $\eta$, this implies \eqref{cf18} for a general $\varphi \in\,H$.
\end{proof}

As the Lebesgue measure on $G$ is invariant for the semigroup $S_\e(t)$,  for any $\e>0$ and $\varphi \in\,C(\bar{G})$
\[\int_G S_\e(t)\varphi(z)\,dz=\int_G\varphi(z)\,dz,\ \ \ \ \ t\geq 0.\]
Now, due to \eqref{fundacor} and the dominated convergence theorem, when we take the limit as $\e$ goes to zero we get
\begin{equation}
\label{cf1001}
\int_G \bar{S}(t)^\vee\varphi(z)\,dz=\int_G\varphi(z)\,dz,\ \ \ \ \ t\geq 0.\end{equation}
Now, if we take $f \in\,C(\bar{\Gamma})$, we have that $f^\vee \in\,C(\bar{G})$, so that, thanks to \eqref{cf1001} and \eqref{cf46}, we get
\begin{equation}
\label{cf1004}
\int_G (\bar{S}(t)f)^\vee(z)\,dz=\int_G (\bar{S}(t)(f^\vee)^\wedge)^\vee(z)\,dz=\int_G f^\vee(z)\,dz.
\end{equation}
Moreover, thanks to \eqref{cf50}, if $\nu$ is the measure defined in \eqref{cf1003}, for any $g \in\,\bar{H}$ we have
\[\int_G g^\vee(z)\,dz=\le<g^\vee,1\r>_{H}=\le<g,1\r>_{\bar{H}}=\int_\Gamma g\,d\nu.\]
Thus, according to \eqref{cf1004}, we can conclude
\[\int_\Gamma \bar{S}(t)f\,d\nu=\int_\Gamma f\,d\nu,\ \ \ \ t\geq 0.\]
This implies the following fact.
\begin{Theorem}
\label{teo30}
The measure $\nu$ is invariant for the semigroup $\bar{S}(t)$. Hence $\bar{S}(t)$ extends to a contraction semigroup on $L^p(\Gamma,\nu)$, for every $p\geq 1$, and in particular on $\bar{H}=L^2(\Gamma,\nu)$.
\end{Theorem} 
As $\bar{S}(t)$ extends to a contraction semigroup on $\bar{H}$, due to Lemma \ref{l1bis} we have  that for any $u \in\,H$
\begin{equation}
\label{cf2000}
|\bar{S}(t)^\vee u|_H=|(\bar{S}(t) u^\wedge)^\vee|_H=|\bar{S}(t) u^\wedge|_{\bar{H}}\leq |u^\wedge|_{\bar{H}}\leq |u|_H.
\end{equation}

Moreover, $\bar{L}$ turns out to be symmetric in $\bar{H}$.
\begin{Lemma}
For any $f,g \in\,D(\bar{L})$, it holds
\[\langle \bar{L}f,g\rangle_{\bar{H}}=\langle f,\bar{L}g\rangle_{\bar{H}}.\]
\end{Lemma}
\begin{proof}
This is an immediate consequence of the boundary conditions imposed on functions in $D(\bar{L})$ and of the definition of the scalar product in $\bar{H}$. Actually, if $f,g \in\,D(\bar{L})$ we have
\[\begin{array}{l}
\ds{\langle \bar{L}f,g\rangle_{\bar{H}}=\sum_{k=1}^N\int_{I_k}{\cal L}_k f(x,k) g(x,k) l_k(x)\,dx=
\frac 12\sum_{k=1}^N\int_{I_k}\frac 1{l_k(x)}\le(l_k f^\prime\r)^\prime (x,k)g(x,k)l_k(x)\,dx}\\
\vs
\ds{=\frac 12\sum_{k=1}^Nl_k f^\prime g_{|_{\partial I_k}}
-\frac 12\sum_{k=1}^N\int_{I_k}l_k(x) f^\prime(x,k)g^\prime (x,k)\,dx=\langle f,\bar{L}g\rangle_{\bar{H}}.}
\end{array}\]

\end{proof}

\section{From the SPDE on the narrow channel $G_\e$ to the SPDE on the domain $G$}
\label{sec3}

We are here interested in the following stochastic reaction diffusion equation in the narrow channel $G_\e$
\begin{equation}
\label{eq-Geps}
\le\{\begin{array}{l}
\ds{\frac{\partial v_\e}{\partial t}(t,x,y)=\frac 12\,\Delta v_\e(t,x,y)+b(v_\e(t,x,y))+\sqrt{\e}\,\frac{\partial w^{Q_\e}}{\partial t}(t,x,y),\ \ \ \ \ (x,y) \in\,G_\e,}\\
\vs
\ds{\frac{\partial  v_\e}{\partial \nu_\e}(t,x,y)=0,\ \ \ (x,y) \in\,\partial G_{\e},\ \ \ \ \ \ \ v_\e(0,x,y)=u_0(x,y\e^{-1}),}
\end{array}\r.
\end{equation}
where $\partial/\partial \nu_\e$ denotes the normal derivative at the boundary of $G_\e$. Here we assume that  $b:\reals\to \reals$ is a Lipschitz-continuous function and $u_0 \in\,C(\bar{G})$. Moreover, we assume that  $w^{Q_\e}(t)$ is a  cylindrical Wiener process taking values in $H_\e=L^2(G_\e)$, having covariance operator $Q_\e^\star Q_\e \in\,{\cal L}_1^+(H_\e)$, that is, for any $t, s\geq 0$ and $u,v \in\,H_\e$ 
\begin{equation}
\label{noise}
\text{\bf E}\le<w^{Q_\e}(t), u\r>_{H_\e}\le<w^{Q_\e}(s), v\r>_{H_\e}=(t\wedge s)\,\le<Q_\e Q_\e^\star u,v\r>_{H_\e}.
\end{equation}
In particular, there exist some complete orthonormal system $\{e^\e_k\}_{k \in\,\nat}$ in $H_\e$ and some sequence of independent standard Brownian motions $\{\beta^\e_k(t)\}_{k \in\,\nat}$, all defined on the same stochastic basis, such that
\[w^{Q_\e}(t)(x,y)=\sum_{k=1}^\infty Q_\e e^\e_k(x,y)\beta^\e_k(t),\ \ \ \ t\geq 0.\]

\medskip

For any $\e_1, \e_2>0$ and $f \in\,H_{\e_1}$, we define
\[J_{\e_2,\e_1}f(x,y)=\sqrt{\frac{\e_1}{\e_2}}\,f(x,\e_1\e_2^{-1} \,y),\ \ \ \ (x,y) \in\,G_{\e_2}.\]
Clearly, $J_{\e_2,\e_1}$ maps $H_{\e_1}$ into $H_{\e_2}$, and for every $\e_1, \e_2, \e_3>0$, we have 
\[J_{\e_3,\e_2}\circ J_{\e_2,\e_1}=J_{\e_3,\e_1}.\]
In particular, $J_{\e_2,\e_1}^{-1}=J_{\e_1,\e_2}$. Moreover
\begin{equation}
\label{scalar}
\le<J_{\e_2,\e_1}\,u,J_{\e_2,\e_1}\,v\r>_{H_{\e_2}}=\le<u,v\r>_{H_{\e_1}}.
\end{equation}
\begin{Lemma}
\label{l1}
Let us fix $\e_1, \e_2>0$. Then, if $\{e_k\}_{k \in\,\nat}$ is a complete orthonormal basis in $H_{\e_1}$, we have that $\le\{J_{\e_2,\e_1}e_k\r\}_{k \in\,\nat}$ is a complete orthonormal basis in $H_{\e_2}$.
\end{Lemma}
\begin{proof}
For any $h,k \in\,\nat$, we have
\[\begin{array}{l}
\ds{\le<J_{\e_2,\e_1}e_k,J_{\e_2,\e_1}e_h\r>_{H_{\e_2}}=\frac{\e_1}{\e_2}\iint_{G_{\e_2}}e_k(x,\e_1\e_2^{-1} y)e_h(x,\e_1\e_2^{-1} y)\,dx\,dy}\\
\vs
\ds{=\iint_{G_{\e_1}}e_k(x, y)e_h(x,y)\,dx\,dy=\le<e_k,e_h\r>_{H_{\e_1}}=\d_{k,h}.}
\end{array}\]
Moreover, as
\[\langle f,J_{\e_2,\e_1} e_k\rangle_{H_{\e_2}}=\langle J_{\e_1,\e_2}f, e_k\rangle_{H_{\e_1}},\]
if $\{e_k\}_{k \in\,\nat}$ is a complete system in $H_{\e_1}$, we have 
\[\langle f,J_{\e_2,\e_1} e_k\rangle_{H_{\e_2}}=0,\ \ \forall\, k \in\,\mathbb{N}\Longleftrightarrow J_{\e_1,\e_2}f=0 \Longleftrightarrow f=0.\]
This implies the completeness of the system $\{J_{\e_2,\e_1} e_k\}_{k \in\,\nat}$.

\end{proof}

Now, for any $\e_1,\e_2>0$ and $Q \in\,{\cal L}(H_{\e_1})$, we define  
\[I_{\e_2,\e_1} Q=J_{\e_2,\e_1}\circ Q\circ J_{\e_1,\e_2} \in\,{\cal L}(H_{\e_2}).\]

\begin{Lemma}
\label{l1.2}
For every $\e_1, \e_2>0$, the operator $I_{\e_2,\e_1}$ is an isometry from ${\cal L}(H_{\e_1})$ into ${\cal L}(H_{\e_2})$  and from ${\cal L}_2(H_{\e_1})$ into ${\cal L}_2(H_{\e_2})$.
\end{Lemma}

\begin{proof}Due to \eqref{scalar},  for any $f \in\,H_{\e_1}$ we have
\[|J_{\e_2,\e_1} f|_{H_{\e_2}}=|f|_{H_{\e_1}},\]
so that $I_{\e_2,\e_1}$ maps ${\cal L}(H_{\e_1})$ into ${\cal L}(H_{\e_2})$ as an isometry.
Moreover, if $\{e^{\e_2}_k\}_{k \in\,\nat}$ is a complete orthonormal system in $H_{\e_2}$, according to Lemma \ref{l1} we have
\[\begin{array}{l}
\ds{\|I_{\e_2,\e_1}Q\|^2_{{\cal L}_2(H_{\e_2})}=\sum_{k=1}^\infty \le|(I_{\e_2,\e_1}Q)e^{\e_2}_k\r|_{H_{\e_2}}^2=\sum_{k=1}^\infty\le|J_{\e_2,\e_1}Q(J_{\e_1,\e_2}e^{\e_2}_k)\r|_{H_{\e_2}}^2}\\
\vs
\ds{=\sum_{k=1}^\infty\le|Q\le(J_{\e_1,\e_2}e^{\e_2}_k\r)\r|_{H_{\e_1}}^2=\|Q\|^2_{{\cal L}_2(H_{\e_1})}.}
\end{array}\]
This proves that $I_{\e_2,\e_1}$ is an isometry from ${\cal L}_2(H_{\e_1})$ into ${\cal L}_2(H_{\e_2})$. \end{proof}

With the above notations, if $v_\e$ is a solution of equation \eqref{eq-Geps} and if we define 
\[u_\e(t,x,y)=\frac 1{\sqrt{\e}}\,(J_{1,\e} v_\e)(t,x,y)=v_\e(t,x,\e y),\ \ \ \ t\geq 0,\ \ (x,y) \in\,G,\]
 we have that 
\begin{equation}
\label{eq1}
\frac{\partial u_\e}{\partial t}(t,x,y)={\cal L}_\e u_\e(t,x,y)+b(u_\e(t,x,y))+\frac{\partial (J_{1,\e}\, w^{Q_\e})}{\partial t}(t,x,y),\end{equation}
where ${\cal L}_\e$ is the uniformly elliptic second order differential operator defined in \eqref{cf303}.

\begin{Lemma}
\label{l1.3}
Assume that there exists some $Q \in\,{\cal L}_2(H)$ such that
\begin{equation}
\label{q-eps}
Q_\e=I_{\e, 1} Q,\ \ \ \ \e>0.
\end{equation}
Then $J_{1,\e}\,w^{Q_\e}(t)\sim w^Q(t)$.
\end{Lemma}
\begin{proof}
According to \eqref{noise} and \eqref{scalar}, for any $t, s \geq 0$  and $u,v \in\,H$ we have
\[\begin{array}{l}
\ds{\text{\bf E}\le<J_{1,\e}\,w^{Q_\e}(t),u\r>_H\le<J_{1,\e}\,w^{Q_\e}(s),v\r>_H=\text{\bf E}\le<w^{Q_\e}(t),J_{\e,1}\,u\r>_{H_\e}\le<w^{Q_\e}(s),J_{\e, 1}\,v\r>_{H_\e}}\\
\vs
\ds{=(t\wedge s)\,\le<Q_\e Q_\e^\star J_{\e,1}\,u,J_{\e,1}\,v \r>_{H_\e}=(t\wedge s)\,\le<J_{1,\e}\,Q_\e Q_\e^\star J_{\e,1}\,u,v \r>_{H}}\\
\vs
\ds{=(t\wedge s)\,\le<(I_{1,\e}\,Q_\e) (I_{1,\e}Q_\e)^\star u,v \r>_{H}.}
\end{array}\]
As we are assuming \eqref{q-eps}, this allows to conclude that $J_{1,\e}\,w^{Q_\e}(t)\sim w^Q(t)$.
\end{proof}

\begin{Remark}
{\em If  $Q \in\,{\cal L}_2(H)$, then there exist a sequence $\{\la_k\}_{k \in\,\nat}$ and a complete orthonormal system $\{e_k\}_{k \in\,\nat}$ in $H$ such that
\[Q e_k=\la_k e_k,\ \ \ \ k \in\,\nat.\]
Then, since in \eqref{q-eps} we assume  $Q_\e=I_{\e,1}Q$, we have
\[Q_\e f^\e_k=\la_k f^\e_k,\ \ \ \ k \in\,\nat,\]
where $\{f^\e_k\}_{k \in\,\nat}$ is the complete orthonormal system of $H_\e$, defined by $f^\e_k=J_{\e,1}e_k$, for any $k \in\,\nat$.}

\end{Remark}

Concerning the boundary conditions satisfied by $u_\e$  we have the following result.
\begin{Lemma}
\label{l1.4}
For any $\e>0$, we have
\begin{equation}
\label{boundary}
\nabla v_\e\cdot \nu^{\e}_{|_{\partial G_\e}}=0\Longleftrightarrow \nabla u_\e\cdot \si_\e \nu_{|_{\partial G}}=0.
\end{equation}
\end{Lemma}

\begin{proof}
According to \eqref{normal}, for any $(x,y) \in\,\partial G$ and $\e>0$, we have
\[\begin{array}{l}
\ds{\nabla u_\e(t,x,y)\cdot \si_\e \nu(x,y)=\frac{\partial u_\e}{\partial x}(t,x,y)\nu_1(x,y)+\frac 1{\e^2}\frac{\partial u_\e}{\partial y}(t,x,y)\nu_2(x,y)}\\
\vs
\ds{=\frac{\partial v_\e}{\partial x}(t,x,\e y)\nu_1(x,y)+\frac 1{\e}\frac{\partial v_\e}{\partial y}(t,x,\e y)\nu_2(x,y)=\frac1{\e\,c_\e(x,y)}\nabla v_\e(t,x,\e y)\cdot \nu^\e(x, \e y).}
\end{array}\]
This implies \eqref{boundary}.

\end{proof}

As a consequence of \eqref{eq1} and Lemmas \ref{l1.3} and \ref{l1.4}, we can conclude that if $v_\e$ is a solution of problem \eqref{eq-Geps}, then $u_\e$ coincides in distribution with  the solution of the problem 
\begin{equation}
\label{eq-G}
\le\{\begin{array}{l}
\ds{\frac{\partial u_\e}{\partial t}(t,x,y)={\cal L}_\e u_\e(t,x,y)+b(u_\e(t,x,y))+\frac{\partial w^{Q}}{\partial t}(t,x,y),\ \ \ \ \ (x,y) \in\,G,}\\
\vs
\ds{\nabla u_\e(t,x,y)\cdot \si_\e \nu(x,y)=0,\ \ \ (x,y) \in\,\partial G,\ \ \ \ \ \ \ u_\e(0,x,y)=u_0(x,y).}
\end{array}\r.
\end{equation}

In what follows, we shall assume that the non-linearity $b:\reals\to\reals$ is Lipschitz-continuous. In particular, this means that the mapping
\[B:H\to H,\ \ \ u \in\,H\mapsto B(u)=b\circ u \in\,H,\]
is well defined and Lipschitz-continuous. Notice that, in the same way, we have that $B:\bar{H}\to\bar{H}$ is well defined and Lipschitz-continuous.

\begin{Definition}
An ${\cal F}_{t}$-adapted process $u_\e \in\,L^p(\Omega;C([0,T];H))$ is a mild solution for problem \eqref{eq-G} if
\[u_\e(t)=S_\e(t)u_0+\int_0^t S_\e(t-s)B(u_\e(s))\,ds+\int_0^t S_\e(t-s)dw^{Q}(s).\]
\end{Definition}
We are assuming here that $Q_\e \in\,{\cal L}_2(H_\e)$, then, according to Lemma \ref{l1.2}, we have that $Q \in\,{\cal L}_2(H)$. This implies that, 
if we define 
\[w_\e(t)=\int_0^t S_\e(t-s)dw^{Q}(s),\]
then $w_\e \in\,L^p(${\boldmath $\Omega$}$;C([0,T];H))$ (in fact this is true also under weaker conditions on $Q$). In particular, since the mapping $B:H\to H$ is Lipschitz-continuous, as a consequence of a fixed point argument in  $L^p(${\boldmath $\Omega$}$;C([0,T];H))$ we can conclude that there exists a unique mild solution $u_\e \in\, L^p(${\boldmath $\Omega$}$;C([0,T];H))$ to equation \eqref{eq-G}.

\section{From the SPDE on  $G$ to the SPDE on the graph $\Gamma$}
\label{sec6}

In Section \ref{sec3}, by a suitable change of variable, from the stochastic reaction-diffusion equation \eqref{eq-Geps} in the narrow channel $G_\e$ we have obtained the following stochastic reaction diffusion equation  in the fixed domain $G$
\begin{equation}
\label{eq-G-bis}
\le\{\begin{array}{l}
\ds{\frac{\partial u_\e}{\partial t}(t,x,y)={\cal L}_\e u_\e(t,x,y)+b(u_\e(t,x,y))+\frac{\partial w^{Q}}{\partial t}(t,x,y),\ \ \ \ \ (x,y) \in\,G,}\\
\vs
\ds{\nabla u_\e(t,x,y)\cdot \si_\e \nu(x,y)=0,\ \ \ (x,y) \in\,\partial G,\ \ \ \ \ \ \ u_\e(0,x,y)=u_0(x,y).}
\end{array}\r.
\end{equation}
 Our purpose here is to study the limiting behavior of its unique mild solution $u_\e$ in the space $L^p(${\boldmath $\Omega$}$;C([0,T];H))$, as $\e\to 0$. 
 
To this purpose, let us consider the problem
\begin{equation}
\label{spde-graph}
\frac{\partial \bar{u}}{\partial t}(t,x,k)=\bar{L} \bar{u}(t,x,k)+b(\bar{u}(t,x,k))+\frac{\partial \bar{w}^{Q}}{\partial t}(t,x,k),\ \ \ \ \bar{u}(0,x,k)=u_0^{\ \wedge}(x,k),
\end{equation}
where $u_0 \in\,C(\bar{G})$ and  $\bar{L}$ is the second order differential operator on $\Gamma$, defined in the interior part of each edge $I_k$ of $\Gamma$ by the operators ${\cal L}_k$, given in \eqref{lk}, and endowed with the gluing conditions described in \eqref{gluing} and \eqref{gluingbis}. Here $\bar{w}^{Q}$ is the cylindrical Wiener process defined by
\begin{equation}
\label{cf22}
\bar{w}^{Q}(t)=\sum_{j=1}^\infty (Qe_j)^\wedge\,\beta_j(t),
\end{equation}
where $\{e_j\}_{j \in\,\nat}$ is a complete orthonormal system in $H$ and $\{\beta_j(t)\}_{j \in\,\nat}$ is a sequence of independent standard Brownian motions. Thanks to \eqref{cf50}, this means that
for any $f,g \in\,\bar{H}$ and $t,s\geq 0$
\[\begin{array}{l}
\ds{\text{\bf E}\langle\bar{w}^{Q}(t),f\rangle_{\bar{H}} \langle\bar{w}^{Q}(s),g\rangle_{\bar{H}} =\sum_{j=1}^\infty \langle (Q e_j)^\wedge,f\rangle_{\bar{H}} \langle (Q e_j)^\wedge,g\rangle_{\bar{H}}(t\wedge s)}\\
\vs
\ds{=\sum_{j=1}^\infty \langle Q e_j,f^\vee\rangle_{H} \langle Q e_j,g^\vee\rangle_{H}(t\wedge s)=\langle Q Q^\star f^\vee,g^\vee\rangle_{H}(t\wedge s)}\\
\vs
\ds{=\langle (QQ^\star f^\vee)^\wedge,g\rangle_{\bar{H}}(t\wedge s)=\langle (QQ^\star)^\wedge  f,g\rangle_{\bar{H}}(t\wedge s).}
\end{array}\]
Notice  that if we assume  $Q \in\,{\cal L}_2(H)$, then, due to Lemma \ref{l2}, we have  
\[\sum_{j=1}^\infty \langle(QQ^\star)^\wedge e_j,e_j\rangle_{\bar{H}}=\sum_{j=1}^\infty \langle QQ^\star e_j^\vee,e_j^\vee\rangle_{H}\leq \|Q\|_{\mathcal{L}_2(H)}^2<\infty,\]
 so that the series in \eqref{cf22} is well defined in $L^2(${\boldmath $\Omega$}$;\bar{H})$, for any $t\geq 0$, and defines a $\bar{H}$-valued Wiener process, with covariance operator $(Q Q^\star)^\wedge$.
 
As we have seen in Section \ref{sec1}, the operator $\bar{L}$ is the generator of the Markov transition semigroup $\bar{S}(t)$ associated with the limiting process 
$\bar{Z}(t)$ defined on the graph $\Gamma$ and introduced in \cite{fw12}. Thus, we can give the following definition.
\begin{Definition}
An adapted process $\bar{u} \in\,L^p(${\boldmath $\Omega$}$;C([0,T];\bar{H}))$ is a {\em mild solution} to equation \eqref{spde-graph} if 
\[\bar{u}(t)=\bar{S}(t)u^{\wedge}_0+\int_0^t \bar{S}(t-s)B(\bar{u}(s))\,ds+\int_0^t\bar{S}(t-s)d\bar{w}^{Q}(s).\]
\end{Definition}

As we are assuming $Q \in\,{\cal L}_2(H)$, then $ \bar{w}^{Q}(t) \in\,L^2(${\boldmath$\Omega$}$,\bar{H})$, for any $t\geq 0$. Moreover, as $\bar{S}(t)$ is a contraction on $\bar{H}$ (see Theorem \ref{teo30}),  the process $w_{\bar{L}}(t)$ defined by
\[w_{\bar{L}}(t):=\int_0^t\bar{S}(t-s)d\bar{w}^{Q}(s),\ \ \ t\geq 0,\] takes values in  $L^p(${\boldmath $\Omega$}$;C([0,T];\bar{H}))$, for any $T>0$ and $p\geq 1$.
Therefore, as the mapping $B:\bar{H}\to \bar{H}$ is Lipschitz-continuous, we have that for any $T>0$ and $p\geq 1$ there exists a unique mild solution $\bar{u} \in\,L^p(${\boldmath $\Omega$}$;C([0,T];\bar{H}))$ to equation \eqref{spde-graph}.

\begin{Theorem}
\label{teo6.2}
Assume that the domain $G$ satisfies assumptions I-IV. Moreover, assume that  the nonlinearity $b:\reals\to\reals$ is Lipschitz-continuous and $Q \in\,{\cal L}_2(H)$. Then, for any $u_0 \in\,C(\bar{G})$, $p\geq 1$  and $0<\tau<T$ we have
\begin{equation}
\label{limfin}
\lim_{\e\to 0}\text{{\em \bf E}}\sup_{t \in\,[\tau,T]}\,|u_\e(t)-\bar{u}(t)^\vee|^p_H=\lim_{\e\to 0}\text{{\em \bf E}}\sup_{t \in\,[\tau,T]}\,|u_\e(t)^\wedge -\bar{u}(t)|^p_{\bar{H}}=0,
\end{equation}
where $u_\e$ and $\bar{u}$ are the unique mild solutions of equations \eqref{eq-G-bis} and \eqref{spde-graph}, respectively.
\end{Theorem}

Before proving \eqref{limfin} in the full generality of Theorem \eqref{teo6.2}, we prove \eqref{limfin} in the case $B=0$ and $u_0=0$.

\begin{Lemma}
Under the same assumption of Theorem \eqref{teo6.2}, for any $T>0$ and $p\geq 1$ we have
\begin{equation}
\label{cf2001}
\lim_{\e\to 0}\text{{\em  \bf E}}\sup_{t \in\,[0,T]}|w_\e(t)-w_{\bar{L}}(t)^\vee|_H^p=0.
\end{equation}
\end{Lemma}
\begin{proof}
For any $t \in\,[0,T]$ and $\a \in\,(0,1/2)$, we have
\[\begin{array}{l}
\ds{\frac{\pi}{\sin \pi \a}\le(w_\e(t)-w_{\bar{L}}(t)^\vee\r)}\\
\vs
\ds{=\int_0^t(t-s)^{\a-1}\bar{S}(t-s)^\vee Y^\e_{\a,1}(s)\,ds+\int_0^t (t-s)^{\a-1}\le[S_\e(t-s)-\bar{S}(t-s)^\vee\r]Y^\e_{\a,2}(s)\,ds,}
\end{array}\]
where
\[Y^\e_{\a,1}(s):=\int_0^s (s-\si)^{-\a}\le[S_\e(s-\si)-\bar{S}(s-\si)^\vee\r]\,dw^Q(\si),\]
and
\[Y^\e_{\a,2}(s):=\int_0^s(s-\si)^{-\a}S_\e(s-\si)dw^Q(\si).\]
Thanks to \eqref{cf2000}, we have that $\bar{S}(t)^\vee$ is a contraction on $H$. Then, that for any $p\geq 1/\a$ it holds
\begin{equation}
\label{cf2003}
\begin{array}{l}
\ds{\text{{ \bf E}}\sup_{t \in\,[0,t]}\,|w_\e(t)-w_{\bar{L}}(t)|_H^p\leq c_{p,\a}(T)\int_0^T\text{{ \bf E}}|Y^\e_{\a,1}(s)|_H^p\,ds}\\
\vs
\ds{+c_{p,\a}(T)\text{{ \bf E}}\sup_{t \in\,[0,T]}\,\int_0^t|\le[S_\e(t-s)-\bar{S}(t-s)^\vee\r]Y^\e_{\a,2}(s)|_H^p\,ds.}
\end{array}\end{equation}

For any fixed $s\geq 0$ and $\a \in\,(0,1/2)$, we have
\[\begin{array}{l}
\ds{\text{{ \bf E}}|Y^\e_{\a,1}(s)|_H^p=c_p\le(\sum_{j=1}^\infty\int_0^s\si^{-2\a}\,\le|S_\e(\si)Qe_j-\bar{S}(\si)^\vee Q e_j\r|_H^2\,d\si\r)^{p/2}}
\end{array}\]
As both $S_\e(\si)$ and $\bar{S}(\si)^\vee$ are contraction in $H$, we have 
\[\begin{array}{l}
\ds{\sum_{j=1}^\infty\le|S_\e(\si)Qe_j-\bar{S}(\si)^\vee Q e_j\r|_H^2\leq c\sum_{j=1}^\infty\,|Qe_j|^2_H,}
\end{array}\]
and then, as $Q \in\,{\cal L}_2(H)$, for any $\eta>0$ we can find $n_\eta \in\,\nat$ such that
\begin{equation}
\label{cf24}
\sum_{j=n_\eta+1}^\infty\int_0^s\si^{-2\a}\le|S_\e(\si)Qe_j-\bar{S}(\si)^\vee Q e_j\r|_H^2\,d\si<\eta.
\end{equation}
Once fixed $n_\eta$,  due to \eqref{cf18} and the dominated convergence theorem we have  that 
\[\lim_{\e\to 0}\,\sum_{j=1}^{n_\eta}\int_0^s\le|S_\e(\si)Qe_j-\bar{S}(\si)^\vee Q e_j\r|_H^2\,d\si=0,\]
and this, together with \eqref{cf24}, due to the arbitrariness of $\eta$ implies that 
\begin{equation}
\label{cf23}
\lim_{\e\to 0}\int_0^T\text{{ \bf E}}|Y^\e_{\a,1}(s)|_H^p\,ds=0.
\end{equation}

Next, for any $0<\tau<t\leq T$ we have
\[\begin{array}{l}
\ds{\int_0^t|\le[S_\e(t-s)-\bar{S}(t-s)^\vee\r]Y^\e_{\a,2}(s)|_H^p\,ds}\\
\vs
\ds{=\int_0^{t-\tau}|\le[S_\e(t-s)-\bar{S}(t-s)^\vee\r]Y^\e_{\a,2}(s)|_H^p\,ds+\int_{t-\tau}^t|\le[S_\e(t-s)-\bar{S}(t-s)^\vee\r]Y^\e_{\a,2}(s)|_H^p\,ds}\\
\vs
\ds{\leq \int_0^{t-\tau}|\le[S_\e(t-s)-\bar{S}(t-s)^\vee\r]Y^\e_{\a,2}(s)|_H^p\,ds+c\,\sqrt{\tau}\le(\int_0^T|Y^\e_{\a,2}(s)|_H^{2p}\,ds\r)^{\frac 12}.}
\end{array}\]
Now, if for any $\d>0$ we denote $k(T,\d)=[T/\d]$, for any $t\leq T$  we have
\[\begin{array}{l}
\ds{\int_0^{t-\tau}|\le[S_\e(t-s)-\bar{S}(t-s)^\vee\r]Y^\e_{\a,2}(s)|_H^p\,ds}\\
\vs
\ds{\leq c_p\sum_{k=1}^{k(T,\d)}\sup_{s \in\,[\tau,T]}|\le[S_\e(s)-\bar{S}(s)^\vee\r]Y^\e_{\a,2}(k\,\d)|_H^p
+c_p\sum_{k=1}^{k(T,\d)}\int_{k\d}^{(k+1)\d}|Y^\e_{\a,2}(s)-Y^\e_{\a,2}(k\,\d)|_H^p\,ds.}
\end{array}\]
This implies that
\[\begin{array}{l}
\ds{\text{{ \bf E}}\sup_{t \in\,[0,T]}\int_0^t|\le[S_\e(t-s)-\bar{S}(t-s)^\vee\r]Y^\e_{\a,2}(s)|_H^p\,ds}\\
\vs
\ds{\leq c_p\sum_{k=1}^{k(T,\d)}\text{{ \bf E}}\sup_{s \in\,[\tau,T]}|\le[S_\e(s)-\bar{S}(s)^\vee\r]Y^\e_{\a,2}(k\,\d)|_H^p+c_p\sum_{k=1}^{k(T,\d)}\int_{k\d}^{(k+1)\d}\text{{ \bf E}}|Y^\e_{\a,2}(s)-Y^\e_{\a,2}(k\,\d)|_H^p\,ds}\\
\vs
\ds{+c\,\sqrt{\tau}\le(\int_0^T\text{{ \bf E}}|Y^\e_{\a,2}(s)|_H^{2p}\,ds\r)^{\frac 12}.}
\end{array}\]
As $Q \in\,{\cal L}_2(H)$, for any $0\leq r<s$, $\e>0$ and $q\geq 1$, we have
\[\text{{ \bf E}}|Y^\e_{\a,2}(s)-Y^\e_{\a,2}(r)|_H^q=c_q\le(\sum_{j=1}^\infty\int_r^s\si^{-2\a}\,\le|S_\e(\si)Qe_j\r|_H^2\,d\si\r)^{q/2}\leq c_q (s-r)^{(1-2\a)\frac q2},\]
so that
\[\begin{array}{l}
\ds{\text{{ \bf E}}\sup_{t \in\,[0,T]}\int_0^t|\le[S_\e(t-s)-\bar{S}(t-s)^\vee\r]Y^\e_{\a,2}(s)|_H^p\,ds\leq c_p(T) \d^{(1-2\a)\frac p2}+c_p(T)\sqrt{\tau}}\\
\vs
\ds{+c_p\sum_{k=1}^{k(T,\d)}\text{{ \bf E}}\sup_{s \in\,[\tau,T]}|\le[S_\e(s)-\bar{S}(s)^\vee\r]Y^\e_{\a,2}(k\,\d)|_H^p.}
\end{array}\]
Therefore, if for any $\eta>0$ if we pick $\d_\eta,\ \tau_\eta>0$ such that 
\[c_p(T) \d_\eta^{(1-2\a)\frac p2}+c_p(T)\sqrt{\tau_\eta}<\frac \eta 2,\]
we have
\[\begin{array}{l}
\ds{\text{{ \bf E}}\sup_{t \in\,[0,T]}\int_0^t|\le[S_\e(t-s)-\bar{S}(t-s)^\vee\r]Y^\e_{\a,2}(s)|_H^p\,ds}\\
\vs
\ds{\leq \frac \eta 2+c_p\sum_{k=1}^{k(T,\d_\eta)}\text{{ \bf E}}\sup_{s \in\,[\tau_\eta,T]}|\le[S_\e(s)-\bar{S}(s)^\vee\r]Y^\e_{\a,2}(k\,\d_\eta)|_H^p.}
\end{array}\]
Thanks to \eqref{cf18} and the dominated convergence theorem, due to the arbitrariness of $\eta>0$ this allows to conclude that
\[\lim_{\e\to 0}\text{{ \bf E}}\sup_{t \in\,[0,T]}\int_0^t|\le[S_\e(t-s)-\bar{S}(t-s)^\vee\r]Y^\e_{\a,2}(s)|_H^p\,ds=0.\]
This, together with \eqref{cf23}, thanks to \eqref{cf2003} allows to get \eqref{cf2001}.

\end{proof}

\begin{proof}[Proof of Theorem  \ref{teo6.2}] As we have seen in Section \ref{sec3}, since $u_\e$ is a mild solution to equation \eqref{eq-G-bis}, we have that $u_\e$ satisfied the following equation
\[u_\e(t)=S_\e(t)u_0+\int_0^t S_\e(t-s)B(u_\e(s))\,ds+w_{\e}(t),\]
where
\[w_{\e}(t)=\int_0^t S_\e(t-s)dw^Q(s).\]
This implies that for any $p\geq 1$ and $T>0$
\begin{equation}
\label{cf40}
\begin{array}{l}
\ds{|u_\e(t)-\bar{u}(t)^\vee|^p_H\leq c_p\,|S_\e(t)u_0-\bar{S}(t)^\vee u_0|_H^p}\\
\vs
\ds{+c_p\,T^{p-1}\int_0^t |S_\e(t-s)B(u_\e(s))-(\bar{S}(t-s)B(\bar{u}(s)))^\vee|_H^p\,ds+c_p\,\le|w_{\e}(t)-w_{\bar{L}}(t)^\vee\r|_H^p.}
\end{array}\end{equation}

Now, for any $\e>0$, $p\geq 1$ and $0\leq s\leq t$, we have
\[\begin{array}{l}
\ds{|S_\e(t-s)B(u_\e(s))-(\bar{S}(t-s)B(\bar{u}(s)))^\vee|_H^p\leq c_p |S_\e(t-s)\le(B(u_\e(s))-B(\bar{u}(s))^\vee\r)|_H^p}\\
\vs
\ds{+c_p\,|S_\e(t-s)B(\bar{u}(s))^\vee-(\bar{S}(t-s)B(\bar{u}(s)))^\vee|_H^p}\\
\vs
\ds{\leq c_p|B(u_\e(s))-B(\bar{u}(s))^\vee|_H^p+c_p\,|S_\e(t-s)B(\bar{u}(s))^\vee-(\bar{S}(t-s)B(\bar{u}(s)))^\vee|_H^p.}
\end{array}\]
Since $B(\bar{u}(s))^\vee=B(\bar{u}(s)^\vee)$, we have
\[|B(u_\e(s))-B(\bar{u}(s))^\vee|_H^p\leq c_p |u_\e(s)-\bar{u}(s)^\vee|_H^p,\]
and then
\begin{equation}
\label{cf30}
\begin{array}{l}
\ds{\int_0^t|S_\e(t-s)B(u_\e(s))-(\bar{S}(t-s)B(\bar{u}(s)))^\vee|_H^p\,ds\leq c_p \int_0^t|u_\e(s)-\bar{u}(s)^\vee|_H^p\,ds}\\
\vs
\ds{+c_p \int_0^t|S_\e(t-s)B(\bar{u}(s))^\vee-(\bar{S}(t-s)B(\bar{u}(s)))^\vee|_H^p\,ds.}
\end{array}
\end{equation}
Therefore, due to \eqref{cf40},
we have
\[|u_\e(t)-\bar{u}(t)^\vee|^p_H\leq R^\e_p(t)+c_p \int_0^t|u_\e(s)-\bar{u}(s)^\vee|_H^p\,ds,\]
where 
\[\begin{array}{l}
\ds{R^\e_p(t):=c_p\,|S_\e(t)u_0-\bar{S}(t)^\vee u_0|_H^p+
c_p\,\le|w_{\e}(t)-w_{\bar{L}}(t)^\vee\r|_H^p}\\
\vs
\ds{+c_p \int_0^t|S_\e(t-s)B(\bar{u}(s))^\vee-(\bar{S}(t-s)B(\bar{u}(s)))^\vee|_H^p\,ds.}
\end{array}\]
By comparison, this yields
\begin{equation}
\label{cf41}
|u_\e(t)-\bar{u}(t)^\vee|^p_H\leq \int_0^t e^{c_p(t-s)}R^\e_p(s)\,ds+c_p\,R_p^\e(t),
\end{equation}
so that
\[\sup_{t \in\,[\tau,T]}|u_\e(t)-\bar{u}(t)^\vee|^p_H\leq c_p(T)\int_0^TR^\e_p(s)\,ds+c_p\,\sup_{t \in\,[\tau,T]}R_p^\e(t).\]
Now, for any $0<\tau<t\leq T$, we have
\[\begin{array}{l}
\ds{\int_0^t|S_\e(t-s)B(\bar{u}(s))^\vee-(\bar{S}(t-s)B(\bar{u}(s)))^\vee|_H^p\,ds}\\
\vs
\ds{\leq c_p\int_0^{t-\tau}|S_\e(t-s)B(\bar{u}(s))^\vee-(\bar{S}(t-s)B(\bar{u}(s)))^\vee|_H^p\,ds+c_p\int_{t-\tau}^t\le(1+|\bar{u}(s)|_{\bar{H}}^p\r)\,ds}\\
\vs
\ds{\leq c_p\int_0^T\sup_{r \in\,[\tau,T]}|S_\e(r)B(\bar{u}(s))^\vee-(\bar{S}(r)B(\bar{u}(s)))^\vee|_H^p\,ds+c_p\sqrt{\tau}\le(\int_0^T\le(1+|\bar{u}(s)|_{\bar{H}}^{2p}\r)\,ds\r)^{\frac 12}.}
\end{array}
\]
This means that
\[\begin{array}{l}
\ds{{\text{\bf E}}\sup_{t \in\,[0,T]}\int_0^t|S_\e(t-s)B(\bar{u}(s))^\vee-(\bar{S}(t-s)B(\bar{u}(s)))^\vee|_H^p\,ds}\\
\vs
\ds{\leq c_p\int_0^T\sup_{r \in\,[\tau,T]}|S_\e(r)B(\bar{u}(s))^\vee-(\bar{S}(r)B(\bar{u}(s)))^\vee|_H^p\,ds+c_p(T){\text{\bf E}}|\bar{u}|_{C([0,T];\bar{H})}^{p}\sqrt{\tau}. }
\end{array}\]
Hence,  for any $\eta>0$ we fix $\tau_\eta>0$ such that
\[\begin{array}{l}
\ds{{\text{\bf E}}\sup_{t \in\,[0,T]}\int_0^t|S_\e(t-s)B(\bar{u}(s))^\vee-(\bar{S}(t-s)B(\bar{u}(s)))^\vee|_H^p\,ds}\\
\vs
\ds{\leq \frac \eta 2+c_p\int_0^T\sup_{r \in\,[\tau_\eta,T]}|S_\e(r)B(\bar{u}(s))^\vee-(\bar{S}(r)B(\bar{u}(s)))^\vee|_H^p\,ds.}
\end{array}\]
As $\eta$ is arbitrary, due to \eqref{cf18} we can conclude that
\[\lim_{\e\to 0}{\text{\bf E}}\sup_{t \in\,[0,T]}\int_0^t|S_\e(t-s)B(\bar{u}(s))^\vee-(\bar{S}(t-s)B(\bar{u}(s)))^\vee|_H^p\,ds=0.\]
This, together with \eqref{cf18} and \eqref{cf2001}, implies that
\[\lim_{\e\to 0}{\text{\bf E}}\sup_{t \in\,[\tau,T]}R^\e_p(t)=0,\]
and due to \eqref{cf41} we can conclude our proof.
 
\end{proof}

\section{From the SPDE on the graph $\Gamma$ to the SPDE on the narrow channel $G_\e$}

Let us consider the equation
\begin{equation}
\label{spde-graph-bis}
\frac{\partial \bar{u}}{\partial t}(t,x,k)=\bar{L} \bar{u}(t,x,k)+b(\bar{u}(t,x,k))+\frac{\partial w^A}{\partial t}(t,x,k),\ \ \ \ \bar{u}(0,x,k)=f_0(x,k),
\end{equation}
where $f_0 \in\,C(\Gamma)$ and  $\bar{L}$ is the second order differential operator on $\Gamma$, introduced in Section \ref{sec1}. Here $w^A(t)$ is a cylindrical Wiener process defined by
\begin{equation}
\label{cf22bis}
w^A(t)=\sum_{j=1}^\infty A f_j\,\beta_j(t),\ \ \ \ t\geq 0,
\end{equation}
where  $\{f_j\}_{j \in\,\nat}$ is a complete orthonormal system in $\bar{H}$, $\{\beta_j\}_{j \in\,\nat}$ is a sequence of independent standard Brownian motions, and $A$ is a bounded linear operator on $\bar{H}$.

Due to \eqref{cf46} and \eqref{cf36}, 
we have
\[A f=(A^\vee f^\vee)^\wedge,\ \ \ f \in\,\bar{H}.\]
Now, as we have seen in Section \ref{sec2}, any $u \in\,H$ can be written as $u=u_1+u_2$, with $u_1 \in\,K_1$ and $u_2 \in\,K_2$. Therefore, if we define
\[Q u=A^\vee u_1,\ \ \ \ u \in\,H,\]
we have that 
\[|Qu|_H=|A^\vee u_1|_H\leq \|A^\vee\|_{{\cal L}(H)}|u_1|_H\leq \|A\|_{{\cal L}(\bar{H})}|u|_H,\]
so that
$Q \in\,{\cal L}(H)$. Moreover if $\{g_i\}_{i \in\,\nat}$ is a complete orthonormal system in $K_2$, we have that $\{h_k\}_{k \in\,\nat}:=\{f_j^\vee\}_{j \in\,\nat}\cup \{g_i\}_{i \in\,\nat}$ is a complete orthonormal system in $H$ and then, thanks to Lemma \ref{l1bis}, we have
\[\begin{array}{l}
\ds{\|Q\|_{{\cal L}_2(H)}^2=\sum_{k \in\,\nat}|Q h_k|_H^2=\sum_{j \in\,\nat}|Q f_j^\vee|_H^2}\\
\vs
\ds{=\sum_{j \in\,\nat}|A^\vee  f_j^\vee|_H^2=\sum_{j \in\,\nat}|(A^\vee f_j^\vee)^\wedge|^2_{\bar{H}}=\sum_{j \in\,\nat}|Af_j|_{\bar{H}}^2=\|A\|^2_{{\cal L}_2(\bar{H})}.}
\end{array}\]
This means that  $A \in\,{\cal L}_2(\bar{H})$ if and only if $Q \in\,{\cal L}_2(H)$.

Thus, we can introduce the stochastic PDE in the fixed domain $G$
\begin{equation}
\label{eq-G-bis-tris}
\le\{\begin{array}{l}
\ds{\frac{\partial u_\e}{\partial t}(t,x,y)={\cal L}_\e u_\e(t,x,y)+b(u_\e(t,x,y))+\frac{\partial w^{Q}}{\partial t}(t,x,y),\ \ \ \ \ (x,y) \in\,G,}\\
\vs
\ds{\nabla u_\e(t,x,y)\cdot \si_\e \nu(x,y)=0,\ \ \ (x,y) \in\,\partial G,\ \ \ \ \ \ \ u_\e(0,x,y)=f_0^\vee(x,y),}
\end{array}\r.
\end{equation}
where $w^Q(t)$ is the $H$-valued Wiener process defined by
\[w^Q(t)=\sum_{k=1}^\infty Q h_k\,\beta^\prime_k(t)=\sum_{j=1}^\infty Qf_j^\vee \beta_j(t),\ \ \ \ t\geq 0,\]
and $\{\beta^\prime_k\}_{k \in\,\nat}:=\{\beta_j\}_{j \in\,\nat}\cup \{\beta^{\prime\prime}_i\}_{i \in\,\nat}$, for some sequence $\{\beta^{\prime \prime}_i(t)\}_{i \in\,\nat}$ of independent Brownian motions, independent of the sequence $\{\beta_j\}_{j \in\,\nat}$. 

Since for any $i,j \in\,\nat$
\[(Q f_j^\vee)^\wedge=(A^\vee f_j^\vee)^\wedge =A f_j,\ \ \ \ \ \ \ Q g_i=0,\]
we have that 
\[w^A(t)= \sum_{k=1}^\infty (Q h_k)^\wedge \,\beta^\prime_k(t).\]
This means that we are exactly in the situation covered by Theorem \ref{teo6.2} and we have that for any $t>0$ and $p\geq 1$
\begin{equation}
\label{cf500}
\lim_{\e\to 0}\text{ \bf E}\,|u_\e(t)-\bar{u}(t)^\vee|^p_H=\lim_{\e\to 0}\text{\bf E}\,|u_\e(t)^\wedge -\bar{u}(t)|^p_{\bar{H}}=0.
\end{equation}

It is important to notice that \eqref{cf500} follows directly from \eqref{limite} and does not require all what we have done in Sections \ref{app1} and \ref{sec4}.

Finally, if we set
\[v_\e(t,x,y)=\sqrt{\e}\,(J_{\e,1} u_\e)(t,x,y)=u_\e(t,x,y/\e),\ \ \ \ (x,y) \in\,G_\e,\]
we have that $v_\e$ satisfies the equation in the narrow channel $G_\e$
\begin{equation}
\label{eq-Geps-tris}
\le\{\begin{array}{l}
\ds{\frac{\partial v_\e}{\partial t}(t,x,y)=\frac 12\,\Delta v_\e(t,x,y)+b(v_\e(t,x,y))+\sqrt{\e}\,\frac{\partial w^{Q_\e}}{\partial t}(t,x,y),\ \ \ \ \ (x,y) \in\,G_\e,}\\
\vs
\ds{\frac{\partial  v_\e}{\partial \nu_\e}(t,x,y)=0,\ \ \ (x,y) \in\,\partial G_{\e},\ \ \ \ \ \ \ v_\e(0,x,y)=f_0^\vee (x,y\e^{-1}),}
\end{array}\r.
\end{equation}
where $\partial/\partial \nu_\e$ denotes the normal derivative at the boundary of $G_\e$ and $Q_\e=I_{\e,1}Q$.

\end{document}